\pdfoutput=1
\documentclass[11pt, a4paper]{article}
\usepackage[a4paper,marginparwidth=2cm,scale=0.75,heightrounded]{geometry}    
\usepackage[utf8]{inputenc}                                                  
\usepackage[main=english]{babel}                                             
\usepackage[autostyle=true]{csquotes}                                        
\usepackage[svgnames]{xcolor}                                                          
\usepackage[perpage]{footmisc}                                               
\usepackage{scrextend}                                                       
\usepackage[final]{microtype}                                                
\usepackage[pdfusetitle]{hyperref}                                           
\usepackage{graphicx,caption,subcaption, tikz-cd, marvosym}                  
\usepackage{amsmath, mathtools, amssymb, amsfonts, multicol, xfrac, slashed} 
\usepackage{array}                                                           
  
\usepackage{thomasmacros} 
\usepackage{dirtytalk}
\DeclareMathOperator{\Tr}{Tr}
\DeclareMathOperator{\ad}{ad}
\DeclareMathOperator{\ph}{ph}

\usepackage[citestyle=alphabetic,
            bibstyle=alphabetic,
            backend=biber,
            maxnames=5,
            url=false,
            urldate=comp,
            doi=true,
            isbn=false,
            giveninits=true,
            eprint=false,
            noerroretextools]{biblatex}

\bibliography{literatur.bib}
\AtEveryBibitem{\clearfield{month}}
\AtEveryBibitem{\clearfield{day}}
\AtEveryBibitem{\clearlist{language}}
\usepackage[shortlabels]{enumitem}
\setlist[itemize]{noitemsep}
\setlist[enumerate]{itemsep=0.1em, label=\upshape(\arabic*), ref=\arabic*,}
\newlist{myenumi}{enumerate}{1}
\setlist[myenumi,1]{itemsep=0.1em, label=\upshape(\roman*)}
\newlist{myenuma}{enumerate}{1}
\setlist[myenuma,1]{itemsep=0.1em, label=\upshape(\alph*)}
\usepackage{titlesec}
\titleformat*{\section}{\large\bfseries}
\titleformat*{\subsection}{\bfseries}
\usepackage{amsthm}
\usepackage{thmtools}
\declaretheorem[name=Theorem,numberwithin=section]{thm}
\declaretheorem[name=Theorem, numbered=no]{thm*}

\declaretheorem[name=Claim,numbered=no]{claim*}
\declaretheorem[name=Lemma,numberlike=thm]{lem}
\declaretheorem[name=Corollary,numberlike=thm]{cor}
\declaretheorem[name=Proposition,numberlike=thm]{prop}
\declaretheorem[name=Definition,numberlike=thm, style=definition]{defi}
\declaretheorem[name=Example, numberlike=thm, style=remark]{ex}
\declaretheorem[name=Facts and Notation,style=remark,numberlike=thm]{facts}
\declaretheorem[name=Remark, numberlike=thm, style=remark]{rem}
\declaretheorem[name=Notation,numberlike=thm, style=remark]{nota}
\declaretheorem[name=Question,numberlike=thm]{que}
\declaretheorem[name=Theorem]{prethm}

\declaretheorem[name=Corollary,numberlike=prethm]{precor}
\declaretheorem[name=Example,numberlike=prethm]{preex}

\numberwithin{equation}{section}
\allowdisplaybreaks[4]
\usepackage[noabbrev]{cleveref} %
\crefname{figure}{Figure}{Figures}
\crefname{fact}{Fact}{Fact}
\crefname{nota}{Notation}{Notations}
\crefname{claim}{Claim}{Claims}
\crefname{table}{Table}{Tables}
\crefname{thm}{Theorem}{Theorems}
\crefname{lem}{Lemma}{Lemmas}
\crefname{defi}{Definition}{Definitions}
\crefname{cor}{Corollary}{Corollaries}
\crefname{prop}{Proposition}{Propositions}
\crefname{ex}{Example}{Examples}
\crefname{rem}{Remark}{Remarks}
\crefname{que}{Question}{Questions}
\crefname{section}{Section}{Sections}
\crefname{chapter}{Chapter}{Chapters}
\crefname{appendix}{Appendix}{Appendices}
\crefdefaultlabelformat{#2\textup{#1}#3}
\creflabelformat{enumi}{(#2#1#3)}
\crefname{prethm}{Theorem}{}
\crefname{precor}{Corollary}{}
\crefname{prelem}{Lemma}{}
\crefname{preex}{Example}{}
\crefdefaultlabelformat{#2\textup{#1}#3}
\creflabelformat{enumi}{(#2#1#3)}

\usepackage{xparse}
\usepackage{autonum} 
\usepackage{titling}
\title{\Large\textbf{Scalar-rigid submersions are Riemannian products}}
\author{Oskar Riedler and Thomas Tony}
\preauthor{\vspace{-1em}\large\centering}
\postauthor{\par\vspace{1em}}  
\date{}  
\predate{}  
\postdate{} 
\usepackage{xurl}
\hypersetup{
  breaklinks=true,
  colorlinks=true,
  allcolors=blue,
  pdfauthor=Oskar Rieder and Thomas Tony,
}
\AtEndDocument{%
  \bigskip
  \footnotesize
  \textsc{Universit\"at Potsdam, Institut f\"ur Mathematik, 14476 Potsdam, Germany}\par
  Email address:~\href{mailto:oskar.riedler@uni-potsdam.de}{oskar.riedler@uni-potsdam.de},\quad URL:~\href{https://sites.google.com/view/oskarriedler/}{www.sites.google.com/view/oskarriedler/}\par
  Email address:~\href{mailto:math@ttony.eu}{math@ttony.eu},\quad URL:~\href{https://ttony.eu}{www.ttony.eu}
}
\newcommand{\DL}{\Dirac_{\mathcal{L}}}
\newcommand{\ue}{u_\epsilon}

\DeclareMathOperator{\hideg}{deg_{hi}}

\newcommand{\slfrac}[2]{\left.#1\middle/#2\right.}
\newcommand\restrict[2]{#1\raisebox{-.2ex}{$|$}\raisebox{-0.6ex}{}_{#2}}
\newcommand*{\keywordbullet}{\,\textcolor{darkgray}{{\tiny\,\textbullet\,}\,}}
\providecommand{\keywords}[5]
{
  \scriptsize
  \textbf{\textit{Keywords---}} #1\keywordbullet #2\keywordbullet #3\keywordbullet #4\keywordbullet #5
}
\begin{document}
\maketitle 
\begin{abstract}
  Scalar-rigid maps are Riemannian submersions by works of Llarull, Goette--Semmelmann, and the second named author. In this article we show that they are essentially Riemannian products of the base manifold with a Ricci-flat fiber. As an application we obtain a Llarull-type theorem for non-zero degree maps onto products of manifolds of non-negative curvature operator and positive Ricci curvature with some enlargeable manifold. The proof is based on spin geometry for Dirac operators and an analysis connecting Clifford multiplication with the representation theory of the curvature operator.
\end{abstract}
%
\keywords{scalar curvature}{comparison geometry}{Dirac operator}{Riemannian submersion}{curvature operator}
\normalsize
\section{Introduction and main results} \label{sec:introduction}
A celebrated result in scalar curvature comparison geometry is the rigidity theorem of \textcite{Llarull1998}. It states that every area non-increasing smooth non-zero degree map $f\colon M\to S^n$ from a closed connected spin manifold~$M$ of dimension~$n\geq 3$ with $\scal_M\geq n\brackets{n-1}$ onto the round sphere is necessarily an isometry. This theorem provides insight into the general question of how a lower bound on the scalar curvature of a Riemannian manifold constrains its global geometry. \medbreak
\textcite{Llarull1996} generalized his classical rigidity theorem to a class of comparison maps between manifolds of different dimensions. Namely, every area non-increasing smooth map~$f\colon M\to S^n$ of non-zero $\hat{A}$-degree, where~$M$ is a closed connected spin manifold of dimension~$m$, with $\scal_M\geq n\brackets{n-1}$ is already a Riemannian submersion with $\scal_M=n\brackets{n-1}$. \textcite[Theorem~2.4]{Goette2000} replaced the round sphere by some closed connected Riemannian manifold~$N$ of non-negative curvature operator, non-zero Euler characteristic and $\scal_N >2\Ric>0$. The second named author \cite[Theorem~A]{Tony2025a} generalized this rigidity theorem by replacing the condition on the $\hat{A}$-degree by a less restrictive condition involving the higher mapping degree. \medbreak

The previous discussion shows that scalar-rigid maps must be Riemannian submersions. While the class of Riemannian submersions admits a wide range of geometrically rich and interesting examples, no scalar-rigid maps are known beyond those that are essentially Riemannian products. The main purpose of this article is to show that this is always the case (\cref{MainThm}). This generalizes the product rigidity result of \cite[Theorem 0.2]{Goette2001}, where product rigidity is shown for comparison maps with target $\CP^n$. \medbreak
\begin{prethm} \label{MainThm}
  Let~$(N, g_N)$ be a closed connected Riemannian manifold of non-negative curvature operator with $\Ric_N>0$, and let $f\colon M\to N$ be a smooth spin map from a closed connected Riemannian manifold~$(M, g_M)$ satisfying 
  \begin{equation} \label{EqIndexTheoreticConditionThmA}
    \chi\brackets{N}\cdot \hideg\brackets{f}\neq 0 \in \KO_*\brackets{\Cstar \pi_1\brackets{M}}.
  \end{equation}
  Suppose $\scal_M\geq \scal_N\circ f$ and $g_M\geq f^*g_N$. Then there exists a closed connected Ricci-flat manifold~$\brackets{F,g_F}$ such that 
  \addtocounter{equation}{1}
  \begin{equation}
    \brackets{M,g_M} \cong \brackets{N,g_N}\times \brackets{F,g_F} \quad \text{(locally isometric)}
    \tag{\theequation}
  \end{equation}
  and~$f$ is given by the projection onto the first factor. Moreover, if $\scal_N>2\Ric_N>0$, then the rigidity also holds if we assume $g_M\geq f^*g_N$ just on~$\Lambda^2 TM$.
\end{prethm}
In \cref{MainThm}, two Riemannian manifolds are called \textit{locally isometric} if their universal coverings endowed with the lifted Riemannian metrics are isometric, and the map~$f$ is called \textit{spin}, if the first and second Stiefel–Whitney classes satisfy $w_1\brackets{TM}=f^*w_1\brackets{TN}$ and $w_2\brackets{TM}=f^*w_2\brackets{TN}$. Its \textit{higher mapping degree} is defined as the index of the spin Dirac operator of the fiber over a regular value of the map~$f$ twisted by the Mishchenko--Fomenko bundle of the manifold~$M$ (\cref{DefHigherMappingDegree}). This is a generalization of the classical $\hat{A}$-degree.\medbreak
\begin{ex}[Local Riemannian product] \label{ex: local-riemannian-product}The conclusion that $M$ is locally isometric to a Riemannian product means that the universal cover is a Riemannian product. If we let $M=S^{2n}\times[0,1]/\sim$ with $(x,0)\sim(-x,1)$, then the map
$$f:(S^{2n}\times[0,1]/\sim\,, g_{\text{round}}+dt^2)\to(\RP^{2n}, g_{\text{round}}),\qquad [(x,t)]\mapsto [x]$$
is an example of a map satisfying the conditions of \cref{MainThm} which is only locally isometric to a Riemannian product, and not globally.
\end{ex}
\begin{rem}[Existence of a non-trivial parallel spinor] \label{rem:parallel-spinor}
  The Ricci-flat manifold~$F$ in \cref{MainThm} carries a non-trivial parallel spinor on a finite Riemannian covering. This is a priori a stronger statement than in \cref{MainThm}, even though it is conjectured that all Ricci-flat spin manifolds carry a non-trivial parallel spinor on a finite Riemannian covering. See the discussion in \cite{Tony2025b}. In the setting of \cref{MainThm}, we have~$F=f^{-1}\brackets{p}$ for some regular value~$p$ of~$f$. Since the higher mapping degree of~$f$ does not vanish, it follows that the Rosenberg index of~$F$ is non-trivial \cite[cf.][Proposition~2.7]{Tony2025a}. Together with the fact that the fiber~$F$ is Ricci-flat, we obtain the existence of such a non-trivial parallel spinor by \cite[Theorem~A]{Tony2025b}.
\end{rem}
\begin{rem}[Non-negative curvature operator] \label{rem:non-negative-curv-op}
The category of scalar-rigid manifolds $N$ for which we formulate \cref{MainThm} is contained in the class of closed manifolds with non-negative curvature operator and positive Ricci curvature. By the work of many mathematicians such manifolds are classified \cite{GallotMeyer1975, ChenTian2006, BoehmWilking2006}. They are locally isometric to Riemannian products of symmetric spaces of compact type, certain K\"ahler metrics on $\CP^n$ biholomorphic to the standard complex structure, and spheres with metrics of non-negative curvature operator.
\end{rem}
As an application of the product rigidity in \cref{MainThm}, we establish a generalization of Llarull's rigidity theorem \cite{Llarull1998} to non-zero degree maps to products of the form $N\times F$ for certain~$F$ including enlargeable manifolds (\cref{cor:Llarull-type-general}). This generalization is in the spirit of \cite[Theorem 1.1]{Hao2024}\,---\,which treats non-zero degree maps of the form $M^n\to S^3\times T^{n-3}$ with $4\leq n\leq 7$\,---\,and generalizes a recent rigidity theorem by \textcite[Theorem~A]{Chow2025} to all even dimensions (\cref{ex:maps-to-SxT}). While the approach in this article is purely based on the Dirac operator method, the proofs by \textcite{Hao2024} and \textcite{Chow2025} rely on minimal slicing combined with spinorial techniques and are therefore restricted to low dimensions.\medbreak
\begin{precor} \label{cor:Llarull-type-general}
  Let $f\colon M\to N\times F$ be a smooth map of non-zero degree, where 
  \begin{itemize}
    \item $M$ and~$N$ are closed connected Riemannian spin manifolds,
    \item $N$ has non-zero Euler characteristic, non-negative curvature operator and $\Ric_N>0$, and
    \item $F$ is a closed connected oriented manifold which is either (area)-enlargeable, or rationally essential and whose fundamental group satisfies the strong Novikov conjecture. 
  \end{itemize}
  If $g_M\geq \brackets{\pr_1\circ f}^*g_N$ and $\scal_M\geq \scal_N\circ \pr_1\circ f$, then~$M$ is locally isometric to the Riemannian product $N\times T^k$ for a flat torus~$T^k$ and~$\pr_1\circ f$ is given by the projection onto the first factor. 
  \end{precor}
Here \textit{rationally essential} means that the fundamental class of~$F$ maps non-trivially under the map $H_*\brackets{F;\Q} \to H_*\brackets{B\pi_1\brackets{F}}$ induced by the classifying map $\mu\colon F\to B\pi_1\brackets{F}$ of the universal cover. We say that its fundamental group satisfies the \textit{strong Novikov conjecture} if the Novikov assembly map is rationally injective. \medbreak
By the assumptions on the manifold~$F$, the map $\pr_1\circ f\colon M\to N$ satisfies the condition on the higher mapping degree (see \cref{sec:completion-proof}). \cref{cor:Llarull-type-general} then follows from \cref{MainThm} applied to the map $\pr_1\circ f\colon M\to N$. 
\begin{preex} \label{ex:maps-to-SxT}
  Let~$(M, g_M)$ be a closed connected Riemannian spin manifold of dimension $2n+k$, $(S^{2n}, g_\text{round})$ the round sphere, $(T^k, g_\text{flat})$ a flat torus, and
  $$f\colon M\to S^{2n}\times T^k$$
  an area non-increasing smooth map of non-zero degree. If $\scal_M\geq 2n\brackets{2n-1}$, then~$M=(S^{2n},g_\text{round})\times (T^k,\overline{g}_\text{flat})$ for some flat metric $\overline{g}_\text{flat}$ and the map~$f$ is of the form $f(x,y) = (x, h(y))$ for some $1$-Lipschitz map $h$ of non-zero degree.
\end{preex}
\begin{rem}
  If the manifold~$N$ in \cref{cor:Llarull-type-general} satisfies $\scal_N>2\Ric_N>0$, then the rigidity statement in \cref{cor:Llarull-type-general} remains valid even when the assumptions on the Riemannian metrics are required only on $2$-vectors.
\end{rem}
\begin{rem}
Note that the manifold $F$ in \cref{cor:Llarull-type-general} is not equipped with a metric, and that it need not be spin. The spin condition and the geometric restrictions only enter for the map $\pr_1\circ f$, not for the map $f$.
\end{rem}
\begin{rem} \label{rem:non-trivial-mapping-fundamental-class}
  The rigidity statement in \cref{cor:Llarull-type-general} also holds if~$F$ is closed connected and oriented, and its fundamental class maps non-trivial under the composition
  \begin{equation}
    H_*\brackets{F,\Q}\To{\ph^{-1}} \KO_*\brackets{F}\otimes_\Z \Q \To{\mu_*} \KO_*\brackets{B\pi_1\brackets{F}}\otimes_\Z \Q \To{\nu} \KO_*\brackets{\Cstar \pi_1\brackets{F}} \otimes_\Z \Q.
  \end{equation}
  Here the first map is the inverse of the (homological) Pontryagin character~$\ph$, the second map the induced map in K-homology of $\mu\colon F\to B \pi_1\brackets{F}$, and the third map the rational Novikov assembly map \cite[cf.][Section~2]{Zeidler2020}. Closed connected oriented manifolds which are rationally essential and the strong Novikov conjecture holds for their fundamental group satisfy this property by definition, as well as (area)-enlargeable manifolds by the proofs of \cite[Theorem~1.4]{Hanke2006} and \cite[Theorem~1.3]{Hanke2007}. 
\end{rem}
In the remaining introduction, we briefly discuss the index theoretic condition in \cref{EqIndexTheoreticConditionThmA} in \cref{MainThm} and provide some further examples where it is satisfied. If the condition on the higher mapping degree in \cref{EqIndexTheoreticConditionThmA} holds, the Rosenberg index of the spin manifold $f^{-1}\brackets{p}$ does not vanish for every regular value~$p$ of the map~$f$. The Rosenberg index of a closed connected spin manifold~$F$, denoted by~$\alpha\brackets{F}$, is the higher index of the $\ComplexCl$-linear spin Dirac operator twisted by the Mishchenko--Fomenko bundle of~$F$ \cite{Rosenberg1983}. Its non-vanishing is the most general known index-theoretic obstruction to the existence of positive scalar curvature metrics. It is well-known for many classes of manifolds that they have non-vanishing Rosenberg index. For instance, if the manifold~$F$ is enlargeable or area-enlargeable \cite{Hanke2006, Hanke2007}, admits a metric of non-positive sectional curvature, or if~$F$ is aspherical and the strong Novikov conjecture holds for its fundamental group \cite{Rosenberg1983}. In all these examples the Rosenberg index is even rationally non-trivial \cite[cf.][Example 5.8]{Tony2025a}. It is useful to detect the non-vanishing of $\chi\brackets{N}\cdot \hideg\brackets{f}$ via a condition involving the Rosenberg index, rather than the higher mapping degree. This relation is worked out in general in \cite[Corollary 5.6]{Tony2025a}, and yields the following examples.
\begin{ex}[Index theoretic condition in \cref{MainThm}] \label{IndexTheoreticConditionExamples}
  Let~$M$ and~$N$ be Riemannian manifolds as in \cref{MainThm}, and $F$ a closed connected spin manifold. In the following examples, the index theoretic condition in \cref{EqIndexTheoreticConditionThmA} in \cref{MainThm} holds. See \cite[Corollary~5.6, Example~5.8, Corollary~5.13 and Example~5.10]{Tony2025a}, respectively.
  \begin{enumerate}
    \item \label{IndexTheoreticConditionExamples1}
      $f\colon M\to N$ smooth spin map with non-zero $\hat{A}$-degree and $\chi\brackets{N}\neq 0$.
    \item \label{IndexTheoreticConditionExamples2}
      $\pr_1\colon N\times F\to N$ with $\chi\brackets{N} \cdot \alpha\brackets{F} \neq 0\in \KO_*\brackets{\Cstar \pi_1\brackets{F}}$. This holds for instance if $\chi\brackets{N}\neq 0$ and~$F$ has rationally non-trivial Rosenberg index, e.g.\  $F$ is enlargeable.
    \item \label{IndexTheoreticConditionExamples3}
      $N$ is two-connected and $f\colon M\to N$ is a fiber bundle with typical fiber~$F$ satisfying $\chi\brackets{N} \cdot \alpha\brackets{F} \neq 0\in \KO_*\brackets{\Cstar \pi_1\brackets{F}}$. This holds for instance if $\chi\brackets{N}\neq 0$ and~$F$ has rationally non-trivial Rosenberg index.
    \item \label{IndexTheoreticConditionExamples4}
      The Euler characteristic of~$N$ is odd and $f\colon M\to N$ is a smooth spin map such that there exists a point~$p\in N$ with $f^{-1}\brackets{p}$ isomorphic to a simply connected manifold of non-vanishing Hitchin invariant \cite{Hitchin1974}. Examples of such manifolds are certain exotic spheres of dimension~$1$ and~$2$ modulo~$8$ \cite[Section 4.3]{Hitchin1974}. Since exotic spheres do not admit Ricci-flat metrics, \cref{MainThm} shows in this case that there exists no Riemannian metric~$g_M$ on~$M$ such that $g_M\geq f^*g_N$ and $\scal_M\geq \scal_N\circ f$ hold.
  \end{enumerate}
\end{ex}
\medbreak

We end our discussion by remarking that it is completely unclear, already in the setting of \cite{Goette2000}, to which extent an index theoretic condition is necessary to achieve rigidity. Can the index condition be replaced by a synthetic condition on the fibers of $f$? We formulate a natural condition in the following question:
\begin{que}
Let $(N,g_N)$ be a closed Riemannian manifold with non-negative curvature operator and positive Ricci curvature, and $(M,g_M)$ a closed Riemannian manifold. Let $f\colon(M,g_M)\to(N,g_N)$ be a \emph{submersion} so that
\begin{enumerate}
\item $\scal_M\geq \scal_N\circ f$ and $f$ is $1$-Lipschitz.
\item $f^{-1}(\{p\})$ does not admit a metric of positive scalar curvature for any $p$.
\end{enumerate}
Do there exist examples where $f$ is not a Riemannian submersion, or even where $M$ is not locally isometric to a product over $N$?
\end{que}
\begin{rem}
It is known that without some global topological condition on $f$ the rigidity conclusion is false, see e.g.\ \cite[Theorem B]{BaerZeidler2025}. For this reason we include the condition that $f$ is submersive. See also \cite[page 139]{Gromov2023}, \cite[page 97]{Gromov-arxiv} for a similar question in the setting of scalar curvature extremality.
\end{rem}
\subsection{Structure of article}
  The article is organized as follows:
  In \cref{sec:proof-sketch} we sketch the proof of \cref{MainThm}, highlighting the main ideas in a simpler setting.
  In the preliminaries in \cref{sec:preliminaries} we give the main definitions and properties concerning Riemannian submersions (\cref{subsec:R-sub}) and higher index theory (\cref{subsec: higher-index}).
  In \cref{sec:setup-pointwise-eq} we set up the spinor bundles in which the proof of \cref{MainThm} takes place, and recap the proof that~$f$ is a Riemannian submersion (\cref{subsec:setup-recap}). Then we establish the pointwise inequalities on which our proof is based (\cref{subsec:establishing-pointwise-ineq}).
  In \cref{sec:pointwise-computations} we use these pointwise inequalities. First we show $\Ric_M=f^*\Ric_N$ in \cref{subsec:ricc-curv-in-riem-submersion}. Then we derive in \cref{subsec:pointwise-eq-for-A-and-T} equations for the O'Neill tensors~$A$ and~$T$ involving the curvature operator of~$N$, and in \cref{subsec: curvature-image} we study some properties of the image of this curvature operator. The previous two subsections are combined in \cref{sec: curv-cliff} to show that the mean curvature of $f^{-1}(\{p\})$ vanishes for all $p\in N$.
  The proofs of \cref{MainThm} and \cref{cor:Llarull-type-general} are then carried out in \cref{sec:completion-proof}.
\subsection{Acknowledgments}
We would like to thank Christoph Böhm and Rudolf Zeidler for helpful discussions and their interest in this work. The second named author thanks his advisor Rudolf Zeidler for his guidance and continuous support. \medbreak
\begin{footnotesize}
  Funded by the European Union (ERC Starting Grant 101116001 – COMSCAL)\footnote{
    Views and opinions expressed are however those of the author(s) only and do not necessarily reflect those of the European Union or the European Research Council. Neither the European Union nor the granting authority can be held responsible for them.
    %
    %
    }
  and by the Deutsche Forschungsgemeinschaft (DFG, German Research Foundation) – Project-ID 427320536 – SFB 1442, as well as under Germany’s Excellence Strategy EXC 2044 390685587, Mathematics M\"unster: Dynamics–Geometry–Structure.
\end{footnotesize}
%
%
\section{Sketch of the proof of \cref{MainThm}} \label{sec:proof-sketch}
In this section we give a brief sketch of the main steps proving \cref{MainThm}. In order to keep the presentation simple we assume that the Euler characteristic of~$N$ and the $\hat{A}$-degree of the map~$f$ do not vanish, rather than the more general index theoretic condition in \cref{EqIndexTheoreticConditionThmA} involving the higher mapping degree of the map~$f$. The general case works with similar ideas as this case, using the methods developed in \cite{Tony2025a}.\medbreak
Let~$f\colon M\to N$ be a smooth map as in \cref{MainThm} and assume that the Euler characteristic of~$N$ and the $\hat{A}$-degree of~$f$ do not vanish. We denote by~$m$ and~$n$ the dimensions of the manifolds~$M$ and~$N$, respectively. Since the map~$f$ is spin, the indefinite vector bundle $\brackets{TM\oplus f^*TN,g_M\oplus \brackets{-g_N}}$ admits an indefinite spin structure, hence we obtain a $\ComplexCl_{m,n}$-linear Dirac bundle~$\SpinBdl$ over the manifold~$M$ together with Clifford multiplications
\begin{equation}
  c\colon TM\to \End\brackets{\SpinBdl} \quad \text{and} \quad \overline{c}\colon f^*TN\to \End\brackets{\SpinBdl}.
\end{equation}
See \cref{ExampleSpinMaps} for the construction of~$\SpinBdl$ and the properties of the Clifford multiplications~$c$ and~$\overline{c}$. We extend the Clifford multiplications~$c$ and~$\overline{c}$ to~$\End\brackets{TM}$ and~$\End\brackets{f^*TN}$ respectively by letting
$$c(v\otimes w^\flat) \coloneq c(v)\cdot c(w),\qquad \overline c(f^*\ \overline v\otimes \overline w^\flat)\coloneq \overline c(f^*\overline v)\cdot \overline c(f^*\overline w)$$
for $v,w\in TM$, $\overline v,\overline w\in TN$. In this way we obtain by the identifications $\Lambda^2 TM \cong \mathfrak{so}\brackets{TM} \subset \End\brackets{TM}$ also a Clifford multiplication by 2-vectors (see \cref{Nota-frames-ClMult.byEnd-L}).
As in \cite{Goette2000}, the Dirac operator~$\Dirac$ of~$\mathcal S$ satisfies
\begin{align}
  & \ind\brackets{D} 
  =\chi\brackets{N}\cdot\deg_{\hat{A}}\brackets{f}\neq 0 \quad \text{and} \label{eq:index-theoem-sketch}\\
  & \Dirac^2 
  = \nabla^*\nabla+
  \Underbrace{\tfrac{1}{4}\scal_M+\tfrac{1}{16}\sum_\alpha \brackets{c\brackets{L\overline{\omega}_\alpha}^2 +\overline{c}\brackets[\normalsize]{f^*\overline{L}\overline{\omega}_\alpha}^2}}{\textstyle \geq\, \frac{1}{4}\brackets{\scal_M-\scal_N\circ f}\,\geq \,0}+
  \Underbrace{\sum_\alpha -\brackets{c\brackets{L\overline{\omega}_\alpha}-\overline{c}\brackets[\normalsize]{f^*\overline{L}\overline{\omega}_\alpha}}^2}{\textstyle \geq\, 0.} \label{eq:Lichnerowicz-sketch}
\end{align}
Here~$\set{\overline{\omega}_\alpha}$ denotes a local orthonormal frame of $\Lambda^2 TN$,~$\overline{L}$ is the square root of the curvature operator of~$N$, and~$L\overline{\omega}_\alpha$ is defined to be $\brackets{\Lambda^2 \mathd f}^{T}\overline{L}\overline{\omega}_\alpha$. \medbreak
By \cref{eq:index-theoem-sketch} there exists a non-trivial smooth section~$u$ of the bundle~$S$ satisfying $\Dirac u=0$. As in \cite{Goette2000}, we obtain from the first parts of \cref{eq:Lichnerowicz-sketch} that the section~$u$ is parallel and that the map~$f$ is a \emph{Riemannian submersion} with $\scal_M=\scal_N\circ f$. From the last summand in \cref{eq:Lichnerowicz-sketch} one finds
\begin{equation} \label{eq:c-barc-proof-sketch}
  \brackets{c\brackets{\omega}-\overline{c}\brackets{f^*\overline{\omega}}}u=0
\end{equation}
for all~$\overline{\omega}\in \image \mathcal{R}_N$ with horizontal lift~$\omega$. Equation~(\ref{eq:c-barc-proof-sketch}) is not utilized in \cite{Goette2000}, but this equation is the starting point for our investigation. \medbreak

Our goal will be to show that the Riemannian submersion $f$ is locally a Riemannian product. This is achieved by showing that the so-called O'Neill tensors $A, T$ of $f$ vanish. These tensors essentially describe how much the horizontal and vertical distributions $\mathcal H$ and $\mathcal V$ of $f$ deviate from being parallel.

We find pointwise relations involving the tensors $A, T$ by taking the covariant derivatives of \cref{eq:c-barc-proof-sketch}. For readability, we use $(c-\overline c)(\omega)$ to denote $c(\omega)-\overline c(f^* \Lambda^2\mathd f\,\omega)$ in what follows, here $\omega\in\Lambda^2 TM$. Since $c,\overline c$, and $u$ are parallel calculating the covariant derivatives is relatively simple, one gets:
\begin{lem}
\label{lem:summary-u-proof-sketch}%
  For every point~$p\in M$ there exists a non-trivial $u\in \SpinBdl_p\cong \ComplexCl_{m,n}$ and a linear map $B\colon T_pM\to \End\brackets{\Lambda^2 \mathcal{H}_p}$ such that
  \begin{equation}
    (1)\; c\brackets{T_U^*\cdot\omega + A_U^*\cdot\omega-\omega\cdot A_U^*}u=0, \quad \text{and} \quad
    (2)\; c\brackets{2A_X\cdot\omega}u+\brackets{c-\overline{c}}\brackets{B_X\brackets{\omega}}u=0
  \end{equation}
  for all vertical $U\in T_pM$, all horizontal $X\in T_pM$, and all $\omega\in \Lambda^2 T_pM$ being the horizontal lift of an element in the image of $\mathcal R_N$. Here $\cdot$ denotes the product of linear maps.
\end{lem}
Here $T_U^*$ and $A_X$ are linear maps $\mathcal H_p\to\mathcal V_p$, and we identify $\omega$ with an element of $\End(T_pM)$. See \cref{subsec:R-sub} for the precise definitions and relations of $A, T, A^*, T^*$ that we use here, in particular \cref{def: A-T-variants} and \cref{lemma: derivative-2-form}. In contrast, the precise form of the linear map $B$ will be irrelevant for all further considerations.
\medbreak

We now explain how to use \cref{eq:c-barc-proof-sketch} to upgrade the relations of \cref{lem:summary-u-proof-sketch}. Let $\omega\in\Lambda^2 T_pM$ be the horizontal lift of an element in the image of the curvature operator $\mathcal R_N$. If $\eta\in\End(T_pM)$ is such that $c(\eta)u=0$ we obtain by \cref{eq:c-barc-proof-sketch}
\begin{equation}
(c-\overline c)(\omega)\,c(\eta) u- c(\eta)\,(c-\overline c)(\omega) u = 0.\label{eq: sketch-com}
\end{equation}
A simple computation (cf.\ \cref{lemma: clifford-to-matrix}) using anti-symmetry of $\omega$ however shows that
\begin{equation}
(c-\overline c)(\omega)\,c(\eta)- c(\eta)\,(c-\overline c)(\omega) = -4\,c(\omega\cdot \eta - \eta\cdot\omega),\label{eq: sketch-LA}
\end{equation}
so that one iteratively upgrades the equation $c(\eta)u=0$ to
\begin{equation}
c((\ad_{\omega_1}\circ...\circ\ad_{\omega_\ell})(\eta))u=0\label{eq: sketch-chain}
\end{equation}
for all $\omega_1,...,\omega_\ell$ being horizontal lifts of elements in $\image(\mathcal R_N)_{f(p)}$.
\medbreak

\textbf{Suppose for the time being that $N$ is a round sphere}, or more generally that the curvature operator of $N$ is invertible. One then finds $(c-\overline c)(B_X(\omega))u=0$ from \cref{eq:c-barc-proof-sketch} and we can drop this annoying term. Recalling that $A_X:\mathcal H_p\to\mathcal V_p$ and applying \cref{eq: sketch-chain} to part (2) of \cref{lem:summary-u-proof-sketch} gives
$$0=c((\ad_{\omega_1}\circ...\circ\ad_{\omega_\ell})(A_X\omega_{\ell+1}))u=-c(A_X \cdot (\omega_1\cdot...\cdot \omega_{\ell+1})^T)u$$
for any $\omega_1,...,\omega_{\ell+1}\in\Lambda^2\mathcal H_p\cong\mathfrak{so}(\mathcal H_p)$. We may however write any linear map $\mathcal H_p\to\mathcal H_p$ as a linear combination of products of elements of $\mathfrak{so}(\mathcal H_p)$. In particular we find for any $X,Y\in\mathcal H_p$ that
\begin{equation}
0=c(A_X\cdot YY^T)u=c(A_XY) c(Y)u.\label{eq: sketch-A=0}
\end{equation}
Now if $A_XY\neq0$ then both $c(A_XY)$ and $c(Y)$ are invertible and \cref{eq: sketch-A=0} clearly leads to a contradiction. We find that $A_XY=0$ for all horizontal $X,Y$ and so $A=0$. Part (1) of \cref{lem:summary-u-proof-sketch} then simplifies to $c(T_U^*\cdot\omega)u=0$. We then repeat the previous steps to find $T_U^*X=0$ for all vertical $U$ and horizontal $X$, giving $T=0$. With $A=T=0$ we have then shown that $f$ is locally a Riemannian product (under the assumption that the curvature operator of $N$ is invertible).
\medbreak

\textbf{We now consider the case of a general $N$}. We begin by explaining how to cancel the $(c-\overline c)(B_X(\omega))u$ term in part (2) of \cref{lem:summary-u-proof-sketch}. A computation very similar to equations~(\ref{eq: sketch-com}), (\ref{eq: sketch-LA}), (\ref{eq: sketch-chain}) extends the relation $(c-\overline c)(\omega)u=0$ from $\omega\in(\Lambda^2 \mathd f)^T(\image\mathcal R_N)\subseteq\Lambda^2\mathcal H_p$ to the Lie-algebra generated by this image, cf.\ \cref{prop: pointwise-TA}. Referring to this Lie-algebra as $\mathfrak g$ we may assume that $B_X(\omega)$ takes values in the $\mathfrak g^\perp\subseteq\Lambda^2\mathcal H_p$ (as the part tangential to $\mathfrak g$ will evaluate to $0$ in \cref{lem:summary-u-proof-sketch}).

Now applying $\ad_\omega$-chains (as sketched in equations (\ref{eq: sketch-com}) -- (\ref{eq: sketch-chain}), cf.\ \cref{lemma: chain}) to part (2) of \cref{lem:summary-u-proof-sketch} gives
\begin{equation}\label{eq: sketch-AB}
c((\ad_{\omega_1}\circ ... \circ\ad_{\omega_\ell})(2A_X\cdot\omega))u + (c-\overline c)((\ad_{\omega_1}\circ ... \circ\ad_{\omega_\ell})(B_X(\omega)))u=0
\end{equation}
for all $\omega_1,...,\omega_\ell\in\mathfrak g$ and $\omega$ a horizontal lift of an element of $\image\mathcal R_N$. Recall that
$$(\ad_{\omega_1}\circ ... \circ\ad_{\omega_\ell})(A_X\cdot\omega) = (A_X\cdot\omega)(\omega_1\cdot...\cdot\omega_\ell)^T,$$
i.e.\ the first term in \cref{eq: sketch-AB} transforms under $\ad$-chains as the representation of $\mathfrak g$ on $\mathcal H_p$, whereas the second transforms as the adjoint representation of $\mathfrak g$ on $\mathfrak g^\perp$. It turns out that the conditions $\mathcal R_N\geq0$, $\Ric_N>0$, imply that these two representations are disjoint (cf.\ \cref{prop: p-disjoint}). It is a general fact of representation theory that one may then find a linear combination $\sum\prod\ad_{\omega_i}$ of $\ad$-chains of $\mathfrak g$ so that
$$\sum\prod \ad_{\omega_i}(A_X\cdot\omega)=A_X\cdot\omega,\qquad \sum\prod\ad_{\omega_i}(B_X(\omega))=0,$$
i.e.\ we may separate \cref{eq: sketch-AB} to find
\begin{equation}
0=c((\ad_{\omega_1}\circ ... \circ\ad_{\omega_\ell})(A_X\cdot\omega))u =- c(A_X \cdot (\omega_1\cdot ... \cdot \omega_\ell\cdot\omega)^T)u\label{eq: sketch-Achain2}
\end{equation}
for all $\omega_1,...,\omega_\ell\in\mathfrak g$ and $\omega$ a horizontal lift of an element of $\image\mathcal R_N$.

In order to work with \cref{eq: sketch-Achain2} we must understand the (associative) sub-algebra of $\End(\mathcal H_p)$ generated by $\mathfrak g$. Recall that in the case of $N$ being a round sphere this turned out to be all of $\End(\mathcal H_p)$, which allowed us to conclude $A_XY=0$ for all $X, Y\in\mathcal H_p$. The general situation is slightly more complicated. In general $\mathcal H_p=\bigoplus_i V_i$ decomposes into irreducible representations of $\mathfrak g$. With respect to this decomposition the associative algebra generated by $\mathfrak g$ decomposes as $\bigoplus_i \mathcal A_i\subseteq\bigoplus_i\End(V_i)$. Further:
\begin{lem}\label{lemma: sketch-curvature-type}
For each $\mathcal A_i\subseteq\End(V_i)$ one has either
\begin{enumerate}
\item $\mathcal A_i = \End(V_i)$.
\item There is an $\R$-linear isometry $\iota:V_i\to\C^k$ for some $k$ so that $\mathcal A_i=\iota^*\End_\C(\C^k)$, i.e.\ $\mathcal A_i$ consists of all complex linear maps on $V_i$. The complex unit $I\in\iota^*\End_\C(\C^k)$ then is an element of $\mathfrak g$.
\end{enumerate} 
\end{lem}
The decomposition described in \cref{lemma: sketch-curvature-type} is an abstract consequence of the curvature operator $\mathcal R_N$ having non-degenerate Ricci tensor, cf.\ \cref{prop: curvature-decomp,prop: curvature-types}.\medbreak

Now if $Y\in V_i\subseteq\mathcal H_p$, where $V_i$ is of the first type described in \cref{lemma: sketch-curvature-type}, one can show as before that $A_XY=0$, $T_U^*Y=0$ for all $X\in\mathcal H_p$ and $U\in\mathcal V_p$. If however $V_i$ is of the second type, the equations of \cref{lem:summary-u-proof-sketch} only imply that $g_M(Y, H_p)=0$, where $H_p=\Tr(T)$ is the mean curvature of the fiber $f^{-1}(\{f(p)\})$ at $p$, cf.\ \cref{prop: complex-case}. If we combine both cases the most we can conclude at this point is:

\begin{cor}\label{cor: sketch-H}
For all $p\in M$ one has $H_p=0$, where $H_p$ denotes the mean curvature of the fiber $f^{-1}(\{f(p)\})$.
\end{cor}
\medbreak

When dealing with a Riemannian submersion the global vanishing of the mean curvature of the fibers is a useful property. A standard curvature identity for Riemannian submersions is
\begin{equation}
\Ric_M(X,X)-\Ric_N(dfX,dfX)=-2\|A_X\|^2-\|T^*_{(\cdot)} X\|^2+g_M( \nabla_X H, X)\label{eq: sketch-ricsub}
\end{equation}
for all horizontal vectors $X$ and appropriate norms $\|\,\|^2$, cf.\ \cref{ThmRicciRelationInRiemannianSubmersion}. An inequality like $\Ric_M(X,X)\geq\Ric_N(\mathd fX, \mathd fX)$ for all horizontal $X$ would allow us to conclude $A=0$, $T=0$ from $H=0$. We will now sketch how \cref{eq:c-barc-proof-sketch} actually implies $\Ric_M(X,X)=\Ric_N(\mathd fX,\mathd fX)$ for all $X\in TM$, not necessarily horizontal:
\medbreak
Recall that $u$ is parallel, so that $R^{\SpinBdl}_\omega u=0$ for all $\omega\in\Lambda^2TM$, here
\begin{equation}
R^{\SpinBdl}_\omega = c(\mathcal R_M(\omega))- \overline c(f^*\mathcal R_N(\Lambda^2 \mathd f\,\omega))\label{eq: sketch-curvend}
\end{equation}
is the curvature endomorphism of the spinor bundle $\SpinBdl$. Equations~(\ref{eq:c-barc-proof-sketch}) and (\ref{eq: sketch-curvend}) immediately imply that
$$c\!\left(\mathcal R_M\omega -(\Lambda^2 \mathd f)^T\circ \mathcal R_N \circ \Lambda^2 \mathd f\omega\right)u =  \left(c(\mathcal R_M\omega) - \overline c(f^*\mathcal R_N \Lambda^2 \mathd f\omega)\right)u=0.$$
 Note that $\mathcal R_M -(\Lambda^2 \mathd f)^T\circ \mathcal R_N \circ (\Lambda^2 \mathd f)$ is a curvature operator on $T_pM$ with Ricci curvature $\Ric_M-\mathd f^T\, \Ric_N \mathd f$, which uses $f$ being a Riemannian submersion. If $\omega_\alpha$ is an orthonormal basis of $\Lambda^2 T_pM$ and $X\in T_pM$ a classical contraction identity then yields:
$$0=\sum_\alpha c(\omega_\alpha\cdot X)\,c\!\left(\mathcal R_M\omega_\alpha -(\Lambda^2 \mathd f)^T\circ \mathcal R_N \circ \Lambda^2 \mathd f\omega_\alpha\right)u=- c(\Ric_M X -\mathd f^T \Ric_N\mathd f X) u,$$
which gives
\begin{equation}
\Ric_M X = \mathd f^T \Ric_N \mathd fX\label{eq: sketch-ric=ric}
\end{equation}
for $X\in T_pM$ arbitrary. Combining \cref{eq: sketch-ric=ric} and \cref{cor: sketch-H} with \cref{eq: sketch-ricsub} allows us to conclude $A=0$, $T=0$ also in the general case.
\medbreak

We conclude this section with some remarks on what modifications must be made when we generalize the index theoretic condition from a non-vanishing $\hat{A}$-degree to $\chi(N)\cdot\hideg(f)\neq0$. In this case one twists the bundle~$\SpinBdl$ with the Mishchenko--Fomenko bundle of $M$. The index condition guarantees the existence of a family $\set{\ue}_{\epsilon>0}$ of almost harmonic sections on $\SpinBdl\otimes\mathcal L$ \cite{Tony2025a}, i.e.\ $\Ltwonorm{\ue}=1$ and $\Dirac^\ell\ue$ is $O(\epsilon)$ in an $L^2$ sense for all $\ell\geq 1$. In our proof we use Moser iteration as in \cite{Tony2025a} to gain pointwise (that is $L^\infty$) estimates involving $\ue$ and its derivatives, cf.\ \cref{PropLinftyAlmost}. We then proceed similarly as described in this sketch, with our pointwise equations holding only up to an $O(\epsilon^r)$ term for an appropriate constant~$r$.

\section{Preliminaries} \label{sec:preliminaries}

In this section we review some preliminary notions related to Riemannian submersions as well as facts from higher index theory. The purpose of the first subsection, \cref{subsec:R-sub}, is to fix our notational conventions for the O'Neill tensors of a Riemannian submersion, and to list the relations we require in our proof. The second subsection, \cref{subsec: higher-index}, provides a bare-bones introduction to the higher index theory used in this paper as well as an explanation of the basic examples that are relevant to our setting.

\subsection{Riemannian submersions}\label{subsec:R-sub}

\begin{defi}
Let $(M,g_M)$, $(N,g_N)$ be Riemannian manifolds. A submersive map $f\colon M\to N$ is called a \emph{Riemannian submersion} if for any $p\in M$ the restriction of $\mathd_pf$ to the \emph{horizontal space}  $\ker(\mathd_pf)^\perp\subseteq T_pM$ is an isometry onto $T_{f(p)}N$.
\end{defi}
\begin{defi}
Let $(M,g_M)$, $(N,g_N)$ be Riemannian manifolds and $f\colon M\to N$ a Riemannian submersion. We denote the \emph{vertical distribution} $\ker(\mathd f)$ by $\mathcal V$ and its orthogonal complement, the \emph{horizontal distribution}, by $\mathcal H$. We define the \emph{O'Neill tensors} of $f$:
\begin{enumerate}
\item The \emph{integrability tensor} of $f$ is the linear map $A:\mathcal H\times\mathcal H\to\mathcal V$ defined by
$$g_M(A_XY,U)=\frac12 g_M([X,Y],U)$$
for $X,Y$ horizontal fields and $U$ a vertical field.
\item The \emph{second fundamental form} of $f$ is the linear map $T:\mathcal V\times\mathcal V\to\mathcal H$ defined by
$$g_M(T_U V,X)= g_M(\nabla_U V,X),$$
where $X$ is a horizontal field, $U,V$ are vertical fields, and $\nabla$ the Levi-Civita connection of $M$.
\end{enumerate}
\end{defi}
\begin{rem}
$A$ and $T$ are  in fact tensors, i.e.\ their values at a point $p\in M$ depend only on the values of the relevant fields $X,Y,U,V$ at $p$.
\end{rem}

\begin{thm} \label{Thm-RiemannianSubmersion-ATtensors}
Let $(M,g_M)$, $(N,g_N)$ be Riemannian manifolds and $f\colon M\to N$ a Riemannian submersion. The following are equivalent:
\begin{enumerate}
\item The tensors $A, T$ both vanish identically.
\item The map $f\colon M\to N$ is locally the projection map of a Riemannian product, i.e.\ for $p\in N$ any point there exists a neighborhood $U_p\subseteq N$ and a Riemannian manifold $(F,g_F)$ so that $f^{-1}(U_p) \cong (U_p\times F, g_N+g_F)$ and $f\lvert_{f^{-1}(U_p)}$ is the projection to the first factor.
\item For $(\widetilde M,g_M), (\widetilde N,g_N)$ the universal covers of $(M,g_M)$, $(N,g_N)$ one has $(\widetilde M,g_M)= (\widetilde N\times F, g_N+ g_F)$ for some Riemannian manifold $(F, g_F)$, and the induced map $\widetilde f\colon\widetilde M\to\widetilde N$ is the projection to the first factor.
\item There is a Riemannian manifold $(F,g_F)$ and a group homomorphism $\rho\colon\pi_1(N)\to\mathrm{Isom}(F)$ so that $(M, g_M)\cong\left( (\widetilde N\times F)/\sim, g_N+g_F\right)$ where $(p, q)\sim (\gamma(p),\rho(\gamma)(q))$ for $\gamma\in\pi_1(N)$.
\end{enumerate}
\end{thm}

Since $A, T$ are $(1,2)$ tensors there are many ways in which they can be contracted with other tensors. This can easily lead to notational problems. In order to accommodate such contractions we introduce the following notations for the different $(1,1)$ tensors associated to $A,T$:
\begin{defi}\label{def: A-T-variants}
Let $(M,g)$, $(N,g_N)$ be Riemannian manifolds, $f\colon M\to N$ a Riemannian submersion, and $p\in M$. For $X,U\in T_pM$ horizontal and vertical respectively, we define the tensors $A_X, S_X, T_U^*, A_U^*$ by:
\begin{align}
T_U^*:\mathcal H_p\to\mathcal V_p,\quad X\mapsto g_M(T_U(\cdot), X)^\sharp,& & &S_X:\mathcal V_p\to\mathcal V_p,\quad U\mapsto T_U^*X,\\
A_X:\mathcal H_p\to\mathcal V_p,\quad Y\mapsto A_XY,& & &A_U^*:\mathcal H_p\to\mathcal H_p,\quad X\mapsto g_M(A_X(\cdot),U)^\sharp.
\end{align}
\end{defi}
\begin{rem}Let $X,Y\in T_pM$ be horizontal and $U\in T_pM$ vertical, then
$$g_M(T_U^*X,T_U^*Y) =g_M\!\left(U,\,\frac{S_XS_Y+S_YS_X}2 U\right).$$
\end{rem}
\begin{rem} \label{RemMeanCurvatureAndS}
The tensor $S_X:\mathcal V\to\mathcal V$ is also called the \emph{shape operator} or \emph{Weingarten map} of the fibers in direction $X$. If $b_i$ is an orthonormal basis of $\mathcal H$ then $\Tr(S_{(\cdot)})^\sharp=\sum_i \Tr(S_{b_i})\,b_i$ is the mean curvature of the fibers.
\end{rem}
\begin{thm}[{\cite[Proposition 9.36]{Besse1987}}] \label{ThmRicciRelationInRiemannianSubmersion}
  Let $(M,g_M)$, $(N,g_N)$ be Riemannian manifolds and $f\colon M\to N$ a Riemannian submersion. Then 
  \begin{equation}
    \ric_M\brackets{X,X}=\ric_N\brackets{\mathd f \brackets{X},\mathd f \brackets{X}}-2\norm{A_X}_2^2-\norm{S_X}_2^2+g_M\brackets{\nabla^M_X H,X}
  \end{equation}
  for all horizontal vectors $X\in TM$. Here~$\norm{\placeholder}_2$ denotes the Hilbert-Schmitt norm, and~$H$ denotes the mean curvature of the fibers.
\end{thm}
\cref{ThmRicciRelationInRiemannianSubmersion} together with the expression of the mean curvature in \cref{RemMeanCurvatureAndS} yield the following corollary. 
\begin{cor}\label{cor: ricci-minimal-fibers}
  Let $(M,g_M)$, $(N,g_N)$ be Riemannian manifolds and $f\colon M\to N$ a Riemannian submersion. If 
  \begin{equation}
    \ric_M\brackets{X,X}=\ric_N\brackets{\mathd f \brackets{X},\mathd f \brackets{X}} \quad \text{and} \quad \Tr\brackets{S_X}=0
  \end{equation}
  for all horizontal vectors $X\in TM$, then~$A$ and~$T$ vanish identically.
\end{cor}

We now remark on some relations that allow us to express the tensors $A$ and $T$ via derivatives. \cref{lemma: derivative-2-form} in particular will be used in the main proof to connect Clifford multiplication with the O'Neill tensors $A, T$.

\begin{lem}\label{lemma: derivative-vector}
Let $(M,g), (N,h)$ be Riemannian manifolds and $f\colon M\to N$ a Riemannian submersion and $p\in M$. Let $X,U\in T_pM$ be horizontal and vertical respectively, and $\overline Y$ a section of $TN$ with horizontal lift $Y=df^T\overline Y$. Then
$$\nabla_U Y = T_U^*Y + A_U^*Y,\qquad \nabla_X Y = A_XY+\widehat{\nabla^N_{dfX}\overline Y},$$
where $\widehat{}${} denotes the horizontal lift of a section of $N$.
\end{lem}
\begin{proof}
These are standard calculations making use of the Koszul formula and functoriality of the Lie bracket, see e.g. \cite[Lemmas 1 to 3]{ONeill1966}.
\end{proof}

\begin{lem}\label{lemma: derivative-2-form}
Let $(M,g), (N,h)$ be Riemannian manifolds, $f\colon M\to N$ a Riemannian submersion, $p\in M$. Let $X,U\in T_pM$ be horizontal and vertical respectively and let $\omega$ be the horizontal lift of a section $\overline\omega\in\Cinfty{N,\Lambda^2TN}$. Then
$$\nabla_U(\omega)= -\ad_\omega(T_U^*-(T_U^*)^T+A^*_U),\qquad \nabla_X \omega = - \ad_\omega(A_X-A_X^T)+\widehat{\nabla^N_{\overline X}\overline\omega},$$
where $\widehat{}$  denotes the horizontal lift of a section on $N$, $\ad_\omega$ the commutator with an element of $\mathfrak{so}(T_pM)\cong\Lambda^2T_pM$, and $^T$ the transpose of a linear map.
\end{lem}
\begin{proof}
Recall that $T_U^*$, $A_X$ take values in $\mathcal V_p$ so that
$$-\ad_\omega(T_U^*)=T_U^*\cdot\omega,\quad \ad_\omega((T_U^*)^T)=\omega\cdot(T_U^*)^T,\quad -\ad_\omega(A_X)=A_X\cdot\omega,\quad \ad_\omega(A_X^T)=\omega\cdot A_X^T$$
for any horizontal $\omega$. On the other hand $A_U^*$ is an anti-symmetric linear map $\mathcal H_p\to\mathcal H_p$, so that
$$-\ad_\omega(A_U^*)= A_U^*\cdot\omega +\omega\cdot (A_U^*)^T.$$
Now for any sections $Y,Z$ of $TM$ being horizontal lifts of $\overline X,\overline Y$ one has by \cref{lemma: derivative-vector}
\begin{align}
\nabla_U(Y\otimes Z^\flat) &= (T_U^*Y)\otimes Z^\flat + Y\otimes (T_U^*Z)^\flat + (A_U^* Y)\otimes Z^\flat + Y\otimes (A_U^*Z)^\flat\\
&= T_U^* \cdot (Y\otimes Z^\flat) + (Y\otimes Z^\flat)\cdot (T_U^*)^T+ A_U^*\cdot (Y\otimes Z^\flat) + (Y\otimes Z^\flat)\cdot (A_U^*)^T,
\end{align}
which shows the first relation of the lemma. The second follows in the same way.
\end{proof}

\subsection{$\A$-linear Dirac bundles and the higher mapping degree}\label{subsec: higher-index}
In this section we give a brief introduction to higher index theory. We define Dirac bundles linear over a $\Cstar$-algebra (\cref{DefALinDiracBundle}), their induced Dirac operators, give a summary of their most important properties and study two relevant examples: First, the $\ComplexCl_m$-linear spinor bundle twisted by the Mishchenko--Fomenko bundle (\cref{ExRosenbergIndex}). Second, the $\ComplexCl_{m,n}$-linear Dirac bundle induced from a spin map~$f\colon M\to N$ (\cref{ExampleSpinMaps}). In the end of this section, we define the higher mapping degree of a spin map (\cref{DefHigherMappingDegree}).\medbreak
For background material concerning (graded Real) unital $\Cstar$-algebras, their Real K-theory, and Hilbert $\Cstar$-modules see \cite{Murphy1990, Roe, Lance1995}. For an introduction to bundles of finitely generated projective Hilbert $\A$-modules as well as first-order $\A$-linear  differential operators see \cite{Mishchenko1980,Ebert2016,Schick2005}. The definition of $\A$-linear Dirac bundles \cite[cf.][Definition 2.1]{Cecchini2023} is a straight forward generalization of $\ComplexCl$-linear Dirac bundles as treated in \cite[Chapter~II~$\S7$]{Lawson1989}. See also \cite{Tony2025a} for more details concerning higher index theory, spin maps and the higher mapping degree. 
\begin{defi} \label{DefALinDiracBundle}
  Let~$\A$ be a graded Real unital $\Cstar$-algebra and~$\brackets{M,g}$ a closed Riemannian manifold of dimension~$m$. A \textit{graded Real $\A$-linear Dirac bundle} is a bundle $\SpinBdl \to M$ of finitely generated projective graded Real Hilbert $\A$-modules equipped with an $\A$-valued inner product, a metric connection, and a parallel bundle map $c\colon TM\to \End\brackets{\SpinBdl}$ such that $c\brackets{v}$ is odd, Real, skew-adjoint, $\A$-linear and $c\brackets{v}^2=-g\brackets{v,v}$ for all $v\in TM$. The grading and the Real structure have to be compatible with all other structures. The \textit{$\A$-linear Dirac operator} induced by a graded Real $\A$-linear Dirac bundle~$\SpinBdl$ is defined via
  \begin{equation}
    \Dirac=\Dirac_{\SpinBdl}\colon \Cinfty{M,\SpinBdl} \To{} \Cinfty{M,\SpinBdl}, \quad \Dirac u\coloneqq \sum_{i=1}^m c\brackets{e_i} \nabla_{e_i} u.
  \end{equation}
  Here $e_1,\dots,e_m$ denotes a local orthonormal frame of~$TM$. 
\end{defi}
\begin{facts} \label{FactsAlinDiracOperator}
Let~$M$ be a closed Riemannian manifold,~$\A$ a graded Real unital $\Cstar$-algebra, and $\SpinBdl$ a graded Real $\A$-linear Dirac bundle. For the induced $\A$-linear Dirac operator~$\Dirac$ one has the following:
\begin{enumerate}
  \item $\Dirac$ is a Real odd elliptic formally self-adjoint first-order $\A$-linear differential operator. 
  \item The $\A$-linear Dirac operator~$\Dirac$ extends to a self-adjoint and regular unbounded operator
  \begin{equation}
    \Dirac \colon \Hsp{1}{M,\SpinBdl}\to \Ltwo{M,\SpinBdl}
  \end{equation} 
  \cite[Theorem 1.14]{Ebert2016}. Here $\Ltwo{M,\SpinBdl}$ and $\Hsp{1}{M,\SpinBdl}$ are graded Real Hilbert $\A$-modules defined as the completion of the graded Real pre-Hilbert $\A$-module $\Cinfty{M,\SpinBdl}$ equipped with the $\A$-valued inner products
  \begin{equation}
    \Ltwoscalarproduct{\placeholder}{\placeholder}\coloneqq \int_M \scalarproduct{\placeholder}{\placeholder}_p \intmathd \mu\brackets{p} \quad \text{and} \quad \Hscalarproduct{1}{\placeholder}{\placeholder} \coloneqq \Ltwoscalarproduct{\nabla \placeholder}{\nabla \placeholder}+\Ltwoscalarproduct{\placeholder}{\placeholder},
  \end{equation}
  respectively. Here we write $\scalarproduct{\placeholder}{\placeholder}_p$ for the $\A$-valued inner product in the fiber of~$\SpinBdl$ over a point $p\in M$. In this context, we define $\abs{\placeholder}^2_p\coloneqq \scalarproduct{\placeholder}{\placeholder}_p$, $\norm{\placeholder}^2_p\coloneqq \Anorm[\normalzise]{\scalarproduct{\placeholder}{\placeholder}_p}$, and ($\A$-valued) $\Ltwoblanc$-norms via
  \begin{equation}
    \Ltwoabs{\placeholder}^2\coloneqq \Ltwoscalarproduct{\placeholder}{\placeholder}\in \A \quad \text{and} \quad 
    \Ltwonorm{\placeholder}^2\coloneqq \Anorm[\big]{\Ltwoscalarproduct{\placeholder}{\placeholder}}\in \R.
  \end{equation}
  \item Since the manifold~$M$ is closed, the functional calculus of the graded Real self-adjoint and regular operator $\Dirac\colon \Hsp{1}{M,\SpinBdl}\to \Ltwo{M,\SpinBdl}$ defines a class in the spectral picture of the $0$-th Real $K$-theory group of the $\Cstar$-algebra~$\A$ \cite{Ebert2016}. This is called \textit{the higher index} of~$\Dirac$ and denoted by $\ind \brackets{\Dirac} \in \KO_0\brackets{\A}$. 
\end{enumerate}
\end{facts}
\begin{ex}[$\ComplexCl_m$-linear spinor bundle and the Rosenberg index] \label{ExRosenbergIndex}
  Let~$M$ be a closed connected Riemannian spin manifold of dimension~$m$. The graded Real $\Cstar$-algebra~$\ComplexCl_m$ is defined as the complex Clifford algebra $\ComplexCl_m$  equipped with the norm given in \cite[p.4]{Ebert2016}, and its natural grading and Real structure. We define the following two bundles over the manifold~$M$.
  \begin{itemize}
    \item 
      The \textit{$\ComplexCl_m$-linear spinor bundle~$\SpinBdl M$} is defined as the vector bundle associated to the $\Spin\brackets{m}$-principle bundle with respect to the representation $\Spin\brackets{m}\to \ComplexCl_m$ given by left multiplication. It defines a graded Real $\ComplexCl_m$-linear Dirac bundle.
    \item 
      Let~$\Cstar \pi_1\brackets{M}$ be the maximal group $\Cstar$-algebra of the fundamental group of the manifold~$M$ together with the trivial grading and its natural Real structure. The \textit{Mishchenko--Fomenko} bundle~$\mathcal{L}\brackets{M}$ of the manifold~$M$ is defined via
      \begin{equation}
        \mathcal{L}\coloneqq \mathcal{L}\brackets{M}\coloneqq \widetilde{M} \times_{\pi_1\brackets{M}} \Cstar \pi_1\brackets{M}.
      \end{equation}
      Here $\pi_1\brackets{M}$ acts on~$\Cstar \pi_1\brackets{M}$ by left multiplication and~$\widetilde{M}$ denotes the universal covering of~$M$. The Mishchenko--Fomenko bundle defines a flat bundle of graded Real finitely generated projective Hilbert $\Cstar \pi_1\brackets{M}$-modules \cite{Mishchenko1980}. 
  \end{itemize} 
  We twist the $\ComplexCl_m$-linear spinor bundle with the Mishchenko--Fomenko bundle and obtain a graded Real $\ComplexCl_m\otimes \Cstar \pi_1\brackets{M}$-linear Dirac bundle~$\SpinBdl M_{\mathcal{L}}$. The higher index of the induced $\ComplexCl_m\otimes \Cstar \pi_1\brackets{M}$-linear Dirac operator is called the \textit{Rosenberg index} of the manifold~$M$ \cite{Rosenberg1983}. 
\end{ex}

\begin{ex}[Spin maps and their $\ComplexCl_{m,n}$-linear spinor bundle {\cite[Section 2.2]{Tony2025a}}] \label{ExampleSpinMaps}
  Let~$\brackets{M,g_M}$ and~$\brackets{N,g_N}$ be closed connected Riemannian manifolds of dimension~$m$ and~$n$, respectively, and $f\colon M\to N$ a smooth spin map. i.e.\ the first and second Stiefel Whitney classes satisfy
  \begin{equation}
    \omega_1\brackets{TM}=f^*\omega_1\brackets{TN} \quad \text{and} \quad \omega_2\brackets{TM}=f^*\omega_2\brackets{TN}.
  \end{equation}
  Since~$f$ is spin, the vector bundle~$TM\oplus f^*TN$ equipped with the indefinite metric $g_M\oplus \brackets{-f^*g_N}$ admits an indefinite spin structure. This is a lift of the $\SO\brackets{m,n}$-principle bundle of orthonormal frames of $TM\oplus f^*TN$ to a $\Spin\brackets{m,n}$-principle bundle along the map $\lambda\colon \Spin\brackets{m,n}\to \SO\brackets{m,n}$ \cite{Baum1981}. The $\Spin\brackets{m,n}$-principle bundle gives rise to a graded Real $\ComplexCl_{m,n}$-linear Dirac bundle~$\SpinBdl$, which comes from the representation of $\Spin\brackets{m,n}$ on~$\ComplexCl_{m,n}$ given by left multiplication, together with Clifford multiplications
  \begin{equation}
   c\colon TM\to \End\brackets{\SpinBdl} \quad \text{and} \quad \overline{c}\colon f^*TN\to \End\brackets{\SpinBdl}
  \end{equation} 
  \cite[Section 2.2]{Tony2025a}. Here $\ComplexCl_{m,n}$ is the graded Real $\Cstar$-algebra generated by $\R^m\oplus \R^n$ such that $\brackets{v,w}^*=\brackets{-v,w}$ and
  \begin{equation}
    \brackets{v,w}\brackets{v',w'}+\brackets{v',w'}\brackets{v,w}=-2\scalarproduct{v}{v'}_{\text{euc}}+2\scalarproduct{w}{w'}_{\text{euc}}
  \end{equation}
  for all $\brackets{v,w},\brackets{v',w'}\in \R^m\oplus \R^n$ \cite[see][p.4]{Ebert2016}. Fiberwise, the Clifford multiplications~$c$ and~$\overline{c}$ are induced by the map $\R^m\oplus \R^n \to \End\brackets{\ComplexCl_{m,n}}$ given by left multiplication, hence
  \begin{itemize}
    \item $c\brackets{v}$ is skew-adjoint and $c\brackets{v}^2=-g_M\brackets{v,v}$ for all $v\in TM$, and
    \item $\overline{c}\brackets{w}$ is self-adjoint and $\overline{c}\brackets{w}^2=g_N\brackets{w,w}$ for all $w\in f^*TN$.
  \end{itemize}
  Locally, the bundle~$\SpinBdl$ is isomorphic to the graded tensor product of the $\ComplexCl_{m,0}$-linear spinor bundle~$\SpinBdl M$ and the $\ComplexCl_{0,n}$-linear spinor bundle~$\SpinBdl N$ equipped with the tensor connection.
\end{ex}
\begin{defi}[Higher mapping degree {\cite[Definition 5.2]{Tony2025a}}] \label{DefHigherMappingDegree}
  Let~$f\colon M\to N$ be a spin map between two closed connected Riemannian manifold~$M$ and~$N$. The higher mapping degree of~$f$ is defined as the higher index of the $\ComplexCl_k$-linear Dirac operator of the spin manifold~$M_p=f^{-1}\brackets{p}$ twisted by the Mishchenko--Fomenko bundle of the manifold~$M$:
  \begin{equation}
    \hideg\brackets{f}\coloneqq \ind\brackets[\big]{\Dirac_{\SpinBdl M_p\otimes \restrict{\mathcal{L}\brackets{M}}{M_p}}}\in \KO_0\brackets{\ComplexCl_k\otimes \Cstar \pi_1\brackets{M}}\cong \KO_k\brackets{\Cstar \pi_1\brackets{M}}.
  \end{equation}
  Here $p\in N$ denotes a regular value of the map~$f$, $k\coloneqq \dim\brackets{M}-\dim\brackets{N}$, and $\SpinBdl M_p$ is the $\ComplexCl_k$-linear spinor bundle of the spin manifold~$M_p$. 
\end{defi}
%

\section{Setup and deriving the pointwise estimates} \label{sec:setup-pointwise-eq}
In this section we give the setup for the proof of \cref{MainThm}. We show that a map~$f\colon M\to N$ as in \cref{MainThm} is a Riemannian submersion (\cref{subsec:setup-recap}) and establish pointwise estimates relating the Clifford multiplications of~$M$ and~$N$ with local geometric quantities (\cref{subsec:establishing-pointwise-ineq}). \medbreak
Throughout this section, let~$\brackets{M,g_M}$ and~$\brackets{N,g_N}$ be closed connected Riemannian manifolds of dimension~$m$ and~$n$, respectively, and let~$f\colon M\to N$ be a smooth spin map such that $\chi\brackets{N}\cdot\hideg{f}\neq 0$. Moreover, we assume that the curvature operator of~$N$ is non-negative, $\scal_M\geq \scal_N\circ f$, and either
\begin{itemize}
  \item $g_M\geq f^*g_N$ and $\Ric_N>0$, or 
  \item $g_M\geq f^*g_N$ on $\Lambda^2 TM$ and $\scal_N>2\Ric_N>0$.
\end{itemize}
Here $\chi\brackets{N}$ denotes the Euler characteristic of~$N$ and $\hideg\brackets{f}$ the higher mapping degree of~$f$, recall \cref{DefHigherMappingDegree}.
\subsection{Setup and recap} \label{subsec:setup-recap}
Denote by~$\A$ the graded Real $\Cstar$-algebra $\ComplexCl_{m,n}\otimes \Cstar \pi_1\brackets{M}$ and its norm by~$\norm{\placeholder}_{\A}$. The spin map~$f$ gives rise to a $\ComplexCl_{m,n}$-linear Dirac bundles~$\SpinBdl$ as in \cref{ExampleSpinMaps}. We twist the bundle~$\SpinBdl$ with the flat Mishchenko--Fomenko bundle of~$M$ and obtain an $\A$-linear Dirac bundle $\SpinBdl_\mathcal{L}\coloneqq \SpinBdl \otimes \mathcal{L}\brackets{M}$ together with Clifford multiplications 
\begin{equation}
  c\colon TM\to \End\brackets{\SpinBdl_\mathcal{L}} \quad \text{and} \quad \overline{c}\colon f^*TN\to \End\brackets{\SpinBdl_\mathcal{L}}.
\end{equation}
We denote the induced $\A$-linear Dirac operator by~$\DL$.\medbreak
Let $e_1,\dots,e_m$ be a local $g_M$-orthonormal frame of~$TM$ and $\overline{e}_1,\dots,\overline{e}_n$ a local $g_N$-orthonormal frame of~$TN$ such that there exists non-negative functions $\mu_1,\dots,\mu_{\min\brackets{m,n}}$ with
  \begin{equation}
    \mathd f\brackets{e_i}=
    \begin{cases}
      \mu_i \overline{e}_i &\text{if } i\leq \min\brackets{m,n} \\
      0 & \text{otherwise}.
    \end{cases}
  \end{equation}
This is possible by considering the singular value decomposition of~$\mathd f$. We use the following notation:
\begin{nota} \label{Nota-frames-ClMult.byEnd-L}
  \begin{enumerate}
    \item \label{item1:Nota-frames-ClMult.byEnd-L}
      Let $\set{\omega_\alpha}=\set{e_i\wedge e_j\given 1\leq i<j\leq n}$ and similarly $\set{\overline{\omega}_\beta}$ be local orthonormal frames of~$\Lambda^2 TM$ and~$\Lambda^2 TN$, respectively.
    \item \label{item2:Nota-frames-ClMult.byEnd-L}
      We define Clifford multiplications 
      \begin{equation}
        c\colon \End\brackets{TM}\to \End\brackets{\SpinBdl_\mathcal{L}} \quad \text{and} \quad \overline{c}\colon \End\brackets{f^*TN}\to \End\brackets{\SpinBdl_\mathcal{L}}
      \end{equation}
      via
      \begin{equation}
        c\brackets{A}\coloneqq \sum_{i,j} g_M\brackets{A e_i,e_j} c\brackets{e_i} c\brackets{e_j}
      \end{equation}
      and similarly for~$\overline{c}$. By the identifications
      \begin{equation}
        \Lambda^2 TM \cong \mathfrak{so}\brackets{TM} \subseteq \End\brackets{TM}, e_i\wedge e_j \to e_i\otimes e_j^\flat-e_j\otimes e_i^\flat
      \end{equation}
      and similar for~$f^*TN$, we obtain Clifford multiplications also by $2$-vectors\footnote{The convention here differs from the notation in \cite{Goette2000} by a factor~$2$. We have $c\brackets{e_i\wedge e_j}=2c\brackets{e_i}c\brackets{e_j}$.}. Here~$\mathfrak{so}$ denotes the Lie algebra of the special orthogonal group, consisting of all anti-symmetric endomorphisms.
    \item \label{item3:Nota-frames-ClMult.byEnd-L}
    We denote the curvature operator of~$N$ by~$\mathcal{R}_N$ and its square root by $\overline{L}$, which exists since $\mathcal{R}_N\geq 0$. Moreover, we write
      \begin{equation} \label{EqHorizontalLiftofL}
        L\overline{\omega}_\alpha \coloneqq \sum_\beta g_N\brackets{\overline{L} \overline{\omega}_\alpha, \overline{\omega}_\beta} \omega_\beta.
      \end{equation}
      Note that $L\overline{\omega}=\brackets{\Lambda^2 \mathd f}^{T} \overline{L}\overline{\omega}$ for all $\overline{\omega}\in \Lambda^2 TN$. 
  \end{enumerate}
\end{nota}
Since the curvature operator of~$N$ is non-negative, the Dirac operator~$\DL$ satisfies as in \cite[Section~1b]{Goette2000} the following Schrödinger-Lichnerowicz type formula.
\begin{lem} \label{LemLichnerowicz}
  The Dirac operator~$\Dirac_{\mathcal{L}}$ satisfies the Schr\"odinger-Lichnerowicz type formula
  \begin{equation} \label{EqLichnerowicz}
  \DL^2 
    = \nabla^*\nabla +\frac{\scal_M}{4} +\frac{1}{16}\sum_\alpha \brackets{c\brackets{L\overline{\omega}_\alpha}^2 +\overline{c}\brackets[\normalsize]{f^*\overline{L}\overline{\omega}_\alpha}^2-\brackets{c\brackets{L\overline{\omega}_\alpha}-\overline{c}\brackets[\normalsize]{f^*\overline{L}\overline{\omega}_\alpha}}^2}
  \end{equation}
  with
  \begin{equation} \label{Eq-relation-w-mu}
  \sum_\alpha c\brackets{L\overline{\omega}_\alpha}^2=-2 \sum_{i,j} \mu_i^2 \mu_j^2 \brackets{R^{TN}_{ijji}\circ f} \quad \text{and} \quad 
  \sum_\alpha\overline{c}\brackets[\normalsize]{f^*\overline{L}\overline{\omega}_\alpha}^2=-2\scal_N\circ f.
  \end{equation}
\end{lem}
\begin{proof}
  Since the Mishchenko--Fomenko bundle is flat, we obtain by \cite[Theorem~4.15 and~8.17]{Lawson1989}
  \begin{equation}
    \begin{split}
      \DL^2
      &=\nabla^*\nabla+\frac{\scal_M}{4}-\frac{1}{8} \sum_{i,j,k,l} g_N\brackets[\big]{R^{TN}_{\overline{e}_i,\overline{e}_j}\overline{e}_k,\overline{e}_l} c\brackets{e_i}c\brackets{e_j}\overline{c}\brackets{f^*\overline{e}_k}\overline{c}\brackets{f^*\overline{e_l}}\\
      &=\nabla^*\nabla+\frac{\scal_M}{4}+\frac{1}{8} \sum_{\alpha, \beta} g_N\brackets{\mathcal{R}_N \overline{\omega}_\alpha, \overline{\omega}_\beta} c\brackets{\omega_\alpha} \overline{c}\brackets{f^*\overline{\omega}_\beta}.
    \end{split}
  \end{equation}
  Compared to \cite{Goette2000}, we have a different sign in front of~$\slfrac{1}{8}$. This is because we consider the negative definite metric $-f^*g_N$ in the construction of the twisting bundle $f^*\SpinBdl N$ (see \cref{ExampleSpinMaps}). Using the existence of the square root of the curvature operator of~$N$ as in \cite[Section~1b]{Goette2000} yields \cref{EqLichnerowicz} and \cref{Eq-relation-w-mu}. 
\end{proof}
\begin{lem} \label{LemExistenceAlmostHarmonicSections}
  There exists a family~$\ue$ of \textit{almost $\DL$-harmonic sections} on~$\SpinBdl_\mathcal{L}$, i.e.\ a family $\set{\ue}\subset \Ccinfty{M,\SpinBdl_\mathcal{L}}$ such that
  \begin{equation}
    \Ltwonorm{\ue}=1 \quad \text{and} \quad 
    \Ltwonorm[\normalzise]{\DL^\ell \ue}< \epsilon^\ell
  \end{equation}
  for all $\epsilon>0$ and for all $\ell\geq 1$.
\end{lem}
\begin{proof}
  By the index theoretic assumption on the map~$f$ in \cref{EqIndexTheoreticConditionThmA} and \cite[][Theorem 5.4]{Tony2025a}, the higher index of the twisted Dirac operator~$\DL$ satisfies
  \begin{equation}
    \ind\brackets{\DL}=\chi\brackets{N} \cdot \hideg\brackets{f} \neq 0 \in \KO_*\brackets{\Cstar \pi_1\brackets{M}}.
  \end{equation}
  It follows that~$\DL$ is not invertible, hence \cite[Lemma 3.2]{Tony2025a} provides a family of almost $\DL$-harmonic sections.
\end{proof}
The Schr\"odinger-Lichnerowicz type formula in \cref{LemLichnerowicz} yields:
\begin{lem} \label{LemLtwoEstimatesNabla}
  For every family~$\ue$ of almost $\DL$-harmonic sections on~$\SpinBdl_\mathcal{L}$ the following holds.
  \begin{enumerate}
    \item \label{LemLtwoEstimatesNabla1}
      $\Ltwonorm{\nabla \ue}<\epsilon$ for all $\epsilon>0$. 
    \item \label{LemLtwoEstimatesNabla2}
      For every~$\ell\geq 2$ there exists a constant~$C$ such that $\Ltwonorm{\nabla^\ell \ue}<C \sqrt{\epsilon}$ for all $\epsilon \in \brackets{0,1}$. 
  \end{enumerate}
\end{lem}
\begin{proof}
  Since $\scal_M\geq \scal_N\circ f$, $\mathcal{R}_N\geq 0$, and~$f$ is area non-increasing, we obtain by \cref{LemLichnerowicz}
  \begin{equation} \label{EqNablaSmall}
    \DL^2\geq \nabla^* \nabla + \frac{\scal_M-\scal_N\circ f}{4}\geq \nabla^* \nabla.
  \end{equation}
  Since~$\ue$ is almost $\DL$-harmonic, we obtain $\Ltwonorm{\nabla \ue}<\epsilon$ for all $\epsilon>0$. For fixed $\ell\geq 2$, there exists a constant~$C$ such that 
  \begin{equation}
    \begin{split}
      \norm[\normalsize]{\nabla^\ell \ue}^2_{\Ltwoblanc}
      =\norm[\big]{\Ltwoscalarproduct[\big]{\brackets{\nabla^*}^{\ell-1} \nabla^\ell \ue}{\nabla \ue}}_{\A}
      \leq \Ltwonorm[\big]{\brackets{\nabla^*}^{\ell-1} \nabla^\ell \ue} \Ltwonorm{\nabla \ue}
      < C^2 \epsilon
    \end{split}
  \end{equation}
  for all $\epsilon \in\brackets{0,1}$. Here~$\nabla^*\colon \Cinfty{M,T^*M\otimes \SpinBdl_\mathcal{L}}\to \Cinfty{M,\SpinBdl_\mathcal{L}}$ denotes the $\Ltwoblanc$-adjoint of~$\nabla$. The second step holds by the Cauchy-Schwarz inequality, and in the last step, we used part~(\ref{LemLtwoEstimatesNabla1}) and
  \begin{equation} \label{EqHnormUniformlyBounded}
    \Ltwonorm[\big]{\brackets{\nabla^*}^{\ell-1} \nabla^\ell \ue}
    \leq \Hnorm{2\ell-1}{\ue}
    \leq C_1\sum_{j=0}^{2\ell-1} \Ltwonorm[\normalsize]{\DL^j \ue}
    <2\ell C_1\eqqcolon C^2,
  \end{equation}
  which holds for a suitable constant~$C_1$ by the fundamental elliptic estimate \cite[cf.][Lemma 2.5]{Tony2025a} and almost $\DL$-harmonicity of~$\ue$.
\end{proof}
For $u\in \Cinfty{M,\SpinBdl_\mathcal{L}}$ we define the $\Linftyblanc$-norm via $\Inftynorm{u}\coloneqq \max_{p\in M} \norm{u}_p$. We upgrade as in the proof of \cite[][Lemma 3.5]{Tony2025a} the $\Ltwoblanc$-estimates on~$\DL^\ell \ue$ and~$\nabla^\ell \ue$ to the $\norm{\placeholder}_\infty$-norm by Moser iteration \cite[cf.][Lemma 3.4]{Tony2025a}. 
\begin{prop} \label{PropLinftyAlmost}
  Let~$\ue$ be a family of almost $\DL$-harmonic sections and $r\coloneqq e^{-\brackets{m-1}\ln\brackets{2}/2}$.
  \begin{enumerate}
    \item \label{PropLinftyAlmost1}
      For every $\ell\in \N_0$, the $\Linftyblanc$-norm of $\nabla^\ell \ue$ is uniformly bounded by some constant.
    \item \label{PropLinftyAlmost2}
      For every differential operator~$P\colon \Cinfty{M,\SpinBdl_\mathcal{L}}\to \Cinfty{M,\SpinBdl_\mathcal{L}}$, there exists a constant~$C$ such that $\Inftynorm{P\ue}\leq C\Ltwonorm{P \ue}^r$ for all $\epsilon\in\brackets{0,1}$.
    \item \label{PropLinftyAlmost3}
      For every $\ell\geq 1$ there exist constants~$C_1$ and~$C_2$ such that 
      \begin{equation} \label{EqLinftyAlmost}
        \Inftynorm[\normalsize]{\Dirac^\ell \ue}<C_1 \epsilon^{r} \quad \text{and} \quad \Inftynorm[\normalsize]{\nabla^\ell \ue}<C_2 \epsilon^{\slfrac{r}{2}}
      \end{equation}
      for all~$\epsilon\in \brackets{0,1}$.
    \item \label{PropLinftyAlmost4}
      There exists an $\epsilon_0\in \brackets{0,1}$, so that $\norm{\ue}^2_p>\brackets{2\vol\brackets{M}}^{-1}$ for all $p\in M$ and all $\epsilon \in \brackets{0,\epsilon_0}$.
  \end{enumerate} 
\end{prop}
\begin{proof}
  Let~$\ue$ be a family of almost $\DL$-harmonic sections. The first part holds by a similar consideration as in \cref{EqHnormUniformlyBounded} and the Sobolev embedding. Let~$P$ be some differential operator on~$\SpinBdl_\mathcal{L}$ of order~$k$. For $\epsilon>0$ and some state~$\rho$ of the $\Cstar$-algebra $\A$, we define the function 
  \begin{equation}
    f_{\epsilon, \rho}\colon M\To{} \R^+, \quad p\to \rho\brackets[\big]{\abs[\normalsize]{P \ue}_p^2}^{\slfrac{1}{2}}.
  \end{equation}
  We obtain for all~$p\in M$, as in the proof of \cite[Lemma 3.5]{Tony2025a}, 
  \begin{equation}
    f_{\epsilon,\rho}\brackets{p}\Laplace f_{\epsilon, \rho}\brackets{p}
    \leq 4\rho\brackets[\big]{\abs[\normalsize]{\nabla^*\nabla P \ue}^2_p}\rho\brackets[\big]{\abs[\normalsize]{P \ue}^2_p}
    \leq \alpha f^2_{\epsilon, \rho}\brackets{p}
  \end{equation}
  for some $\alpha>0$ independent of~$\epsilon$ and the state~$\rho$. In the last step, we used that the $\Linftyblanc$-norm of all covariant derivates of~$\ue$ up to order~$k+2$ are uniformly bounded by part~(\ref{PropLinftyAlmost1}). Moser iteration \cite[cf.][Lemma 3.4]{Tony2025a} yields that there exists a constant~$C_3$ such that 
  \begin{equation} \label{EqLinftyLtwoD}
    \Inftynorm{P \ue}\leq C_3 \Ltwonorm{P \ue}^{r}
  \end{equation}
  for all $\epsilon\in (0,1)$, which proves part~(\ref{PropLinftyAlmost2}). Part~(\ref{PropLinftyAlmost3}) follows immediately from part~(\ref{PropLinftyAlmost2}), the almost $\DL$-harmonicity of~$\ue$ and \cref{LemLtwoEstimatesNabla}. \medbreak
  %
  It remains to show part~(\ref{PropLinftyAlmost4}) of the proposition. By \cite[Lemma 3.6]{Tony2025a}, there exists a constant~$C_4$ such that 
  \begin{equation} \label{EqAlmostConstant}
    \norm[\big]{\overline{u}_\epsilon-\abs{\ue}_p^2}_{\A}< C_4  \epsilon^{r} \quad \text{with} \quad \overline{u}_\epsilon \coloneqq \frac{1}{\vol M}\int_M \abs{\ue}^2_q \intmathd \mu\brackets{q}
  \end{equation}
  for all $\epsilon \in \brackets{0,1}$ and all points $p\in M$. For $\epsilon_0\coloneqq \min\set[\big]{1,\brackets{2C_4\vol\brackets{M}}^{\slfrac{-1}{r}}}$ we obtain the desired estimate using the reverse triangle inequality and \cref{EqAlmostConstant}.
\end{proof}
The existence of a family of almost $\DL$-harmonic section~$\ue$ (\cref{LemExistenceAlmostHarmonicSections}) together with \cref{PropLinftyAlmost} yields the following rigidity statement:
\begin{thm}[{\cite[Theorem 2.4 and Theorem~A]{Goette2000, Tony2025a}}] \label{ThmRiemSubmersion}
  The map~$f$ is a Riemannian submersion with $\scal_M=\scal_N\circ f$ and the Schr\"odinger-Lichnerowicz type formula in \cref{LemLichnerowicz} reduces to
  \begin{equation} \label{EqLichnerowiczRiemanSubmersion}
    \DL^2 = \nabla^*\nabla +\frac{1}{16}\sum_\alpha -\brackets{c\brackets{L\overline{\omega}_\alpha}-\overline{c}\brackets{f^*\overline{L}\overline{\omega}_\alpha}}^2.
  \end{equation}
\end{thm}
\begin{proof} 
The classical version assuming a non-trivial $\hat{A}$-degree of~$f$ is proved by \textcite{Goette2000}, and the general version by the second named author in \cite[Section~5]{Tony2025a}. The argument may be summarised as follows. \medbreak
We obtain by \cref{LemLichnerowicz} together with $g_M\geq f^*g_N$ and $\mathcal{R}_N\geq 0$
\begin{align}
  \DL^2
  &\geq \nabla^*\nabla +\frac{\scal_M}{4} -\frac{1}{8} \sum_{i,j} \mu_i^2 \mu_j^2 \brackets{R^{TN}_{ijji}\circ f} -\frac{\scal_N\circ f}{8} \label{EqLichnerowicz3}\\
  &\geq \nabla^*\nabla +\frac{\scal_M-\scal_N\circ f}{4} \label{EqLichnerowicz4}.
\end{align}
The existence of a family of almost $\DL$-harmonic sections~$\ue$ (\cref{LemExistenceAlmostHarmonicSections}) satisfying the estimates in \cref{PropLinftyAlmost} yields together with \cref{EqLichnerowicz4} the extremality statement $\scal_M=\scal_N\circ f$. Similarly, we obtain by \cref{EqLichnerowicz3}
\begin{equation} \label{RelationScalMu}
  \sum_{i,j} \mu_j^2 \mu_j^2 \brackets{R^{TN}_{ijji}\circ f}=\scal_N\circ f,
\end{equation}
which yields as in \cite[Section 1.c]{Goette2000} $\mu_1=\dots=\mu_n=1$. In other words, $f$ is a Riemannian submersion.\medbreak
It remains to show \cref{EqLichnerowiczRiemanSubmersion}. By \cref{Eq-relation-w-mu}, $\scal_M=\scal_N\circ f$, and the relation in \cref{RelationScalMu}, we obtain
  \begin{equation}
    \begin{split}
      \sum_\alpha c\brackets{L\overline{\omega}_\alpha}^2 + \sum_\alpha\overline{c}\brackets{f^*\overline{L}\overline{\omega}_\alpha}^2
      =-2 \sum_{i,j} \mu_i^2 \mu_j^2 \brackets{R^{TN}_{ijji}\circ f} -2\scal_N\circ f
      =-4\scal_M,
    \end{split}
  \end{equation}
  and the Schr\"odinger-Lichnerowicz type formula in \cref{EqLichnerowicz} reduces to \cref{EqLichnerowiczRiemanSubmersion}.
\end{proof}
\subsection{Establishing pointwise inequalities} \label{subsec:establishing-pointwise-ineq}
Recall that by \cref{ThmRiemSubmersion} and \cref{LemExistenceAlmostHarmonicSections} the map $f\colon M\to N$ is a Riemannian submersion, the twisted Dirac operator~$\DL$ satisfies
\begin{equation} \label{EqLichnerowiczRiemanSubmersion2}
  \DL^2 = \nabla^*\nabla +\frac{1}{16}\sum_\alpha -\brackets{c\brackets{L\overline{\omega}_\alpha}-\overline{c}\brackets{f^*\overline{L}\overline{\omega}_\alpha}}^2,
\end{equation}
and there exists a family~$\ue$ of almost $\DL$-harmonic sections of the bundle~$\SpinBdl_\mathcal{L}$ satisfying the properties in \cref{PropLinftyAlmost}. Throughout this section we fix such a family~$\ue$, a point $p\in M$, $r\coloneqq e^{-\brackets{m-1}\ln\brackets{2}/2}$ and write~$\mathcal{V}$ and~$\mathcal{H}$ for the vertical and horizontal projections from~$T_pM$ onto its vertical and horizontal parts, respectively.\medbreak
The goal of this section is to derive the pointwise estimates stated in \cref{Lem-CurvatureOperator-small}, \ref{Lem-c-barc-small}, \ref{lem:nabla_U-omega-small} and~\ref{lem:nabla_X-omega-small}. Although all four statements are purely pointwise, their proofs require that the sections~$\ue$ are locally defined and almost parallel. 
\begin{lem} \label{Lem-CurvatureOperator-small}
  There exists a constant~$C$ such that
  \begin{equation}
    \norm[\big]{R^{\SpinBdl_\mathcal{L}}_{X,Y} \ue}_p<C\,\norm{X}\,\norm{Y}\,\epsilon^{\slfrac{r}{2}}
  \end{equation}
  for all $X,Y\in T_pM$ and all $\epsilon\in \brackets{0,1}$ where $R^{\SpinBdl_\mathcal{L}}$ denotes the curvature endomorphism of~$\SpinBdl_\mathcal{L}$.
\end{lem}
\begin{proof}
  For~$X,Y\in T_pM$ and an $\epsilon\in \brackets{0,1}$, we obtain 
  \begin{equation}
    \begin{split}
      \norm[\big]{R^{\SpinBdl_\mathcal{L}}_{X,Y} \ue}^2_p
      &\leq 3m^2\sum_{i,j} g_M\brackets{X,e_i}^2 g_M\brackets{Y,e_j}^2 \brackets[\big]{\norm{\nabla_{e_i}\nabla_{e_j}\ue}^2_p+\norm{\nabla_{e_j}\nabla_{e_i}\ue}^2_p+\norm[\big]{\nabla_{\sqbrackets{e_i,e_j}}\ue}^2_p}\\
      &\leq C^2\,\norm{X}^2\,\norm{Y}^2\,\epsilon^{r}.
    \end{split}
  \end{equation}
  The second step holds for an appropriate constant~$C$ independent of~$X$, $Y$ and~$\epsilon$ by \cref{PropLinftyAlmost}~(\ref{PropLinftyAlmost3}).
\end{proof}
\begin{lem} \label{Lem-c-barc-small}
  There exists a constant~$C$ such that 
  \begin{equation}
    \norm{\brackets[\big]{c\brackets{\omega}-\overline{c}\brackets{f^*\overline{\omega}}}\ue}_p<C\,\norm{\omega}\, \epsilon^{\slfrac{r}{4}}
  \end{equation}
  for all $\overline{\omega}\in \image\brackets{\mathcal{R}_N}$ with horizontal lift $\omega\in \Lambda^2T_pM$ and for all $\epsilon\in \brackets{0,1}$. 
\end{lem}
\begin{proof}
  Fix some~$\omega$, $\overline{\omega}$ and $\epsilon\in \brackets{0,1}$ as in the lemma. Since $\image{\mathcal{R}_N}=\image{\overline{L}}$, there exists an~$\overline{\eta}\in \brackets[\normalsize]{\ker \overline{L}}^\perp$ such that $\overline{\omega}=\overline{L}\overline{\eta}$. We obtain
  \begin{equation} \label{EqProofDifferenceCupCdown1}
    \overline{\omega}=\sum_\alpha g_N\brackets{\overline{\eta},\overline{\omega}_\alpha} \overline{L} \overline{\omega}_\alpha \quad \text{and} \quad 
    \omega=\sum_\alpha g_N\brackets{\overline{\eta},\overline{\omega}_\alpha} L\overline{\omega}_\alpha,
  \end{equation}
  hence
  \begin{equation} \label{EqProofDifferenceCupCdown2}
    \begin{split}
      \abs{\brackets{c\brackets{\omega}-\overline{c}\brackets{f^*\overline{\omega}}}\ue}_p^2
      &=\abs[\Big]{\sum_\alpha g_N\brackets{\overline{\eta},\overline{\omega}_\alpha}\brackets{c\brackets{L\overline{\omega}_\alpha}-\overline{c}\brackets[\normalsize]{f^*\overline{L}\overline{\omega}_\alpha}}\ue}_p^2\\
      &\leq C_1\norm{\overline{\omega}}^2 \sum_\alpha \abs[\big]{\brackets{c\brackets{L\overline{\omega}_\alpha}-\overline{c}\brackets[\normalsize]{f^*\overline{L}\overline{\omega}_\alpha}}\ue}_p^2 \\
      &= 16C_1\norm{\omega}^2 \brackets{\scalarproduct{\Dirac^2_{\mathcal{L}}\ue}{\ue}_p-\scalarproduct{\nabla^*\nabla \ue}{\ue}_p}.
    \end{split}
  \end{equation}
  The second step holds by the estimate
    \begin{equation}
      \begin{aligned}
      &\abs[\Big]{\sum_{\alpha} g_N\brackets{\overline{\eta},\overline{\omega}_\alpha} \brackets{c\brackets{L\overline{\omega}_\alpha}-\overline{c}\brackets[\normalsize]{f^*\overline{L}\overline{\omega}_\alpha}}u}^2_p 
      &&\leq \tbinom{n}{2} \sum_\alpha g_N\brackets{\overline{\eta},\overline{\omega}_\alpha}^2 \abs[\big]{\brackets{c\brackets{L\overline{\omega}_\alpha}-\overline{c}\brackets[\normalsize]{f^*\overline{L}\overline{\omega}_\alpha}}u}^2_p \\
      &&&\leq \tbinom{n}{2} \norm{\overline{\eta}}^2 \sum_\alpha \abs{\brackets{c\brackets{L\overline{\omega}_\alpha}-\overline{c}\brackets[\normalsize]{f^*\overline{L}\overline{\omega}_\alpha}}u}^2_p\\
      &&&\leq C_1\norm{\overline{\omega}}^2 \sum_\alpha \abs[\big]{\brackets{c\brackets{L\overline{\omega}_\alpha}-\overline{c}\brackets[\normalsize]{f^*\overline{L}\overline{\omega}_\alpha}}u}_p^2\\
      &\text{with} \quad C_1=\tbinom{n}{2} \norm[\big]{\overline{L}\vert^{-1}_{\brackets[\normalsize]{\ker \overline{L}}^\perp}}^2_{\text{op}}, &&&
      \end{aligned}
    \end{equation}
  and the third step by the skew-adjointness of $c\brackets{L\overline{\omega}_\alpha}-\overline{c}\brackets{f^*\overline{L}\overline{\omega}_\alpha}$ and \cref{EqLichnerowiczRiemanSubmersion}. We take the norm in the $\Cstar$-algebra $\A$ of \cref{EqProofDifferenceCupCdown2}, and obtain for a suitable constant~$C$, independent of~$\omega$ and~$\epsilon$,
  \begin{equation}
    \begin{split}
    \norm{\brackets{c\brackets{\omega}-\overline{c}\brackets{f^*\overline{\omega}}}\ue}_p^2
    \leq 16 C_1 \norm{\omega}^2 \brackets{\norm[\normalsize]{\DL^2 \ue}_p \norm{\ue}_p+
    \norm{\nabla^*\nabla \ue}_p \norm{\ue}_p}
    \leq C^2 \norm{\omega}^2 \epsilon^{\slfrac{r}{2}}.
    \end{split}
  \end{equation}
  Here we used the triangle inequality, the Cauchy-Schwarz inequality, and \cref{PropLinftyAlmost}~(\ref{PropLinftyAlmost3}).
\end{proof}
In order to get information about the O'Neill tensors~$A$ and~$T$ of the Riemannian submersion~$f$, we take covariant derivatives of \cref{EqLichnerowiczRiemanSubmersion2} in the direction of vertical and horizontal vectors. We obtain \cref{lem:nabla_U-omega-small} and \cref{lem:nabla_X-omega-small} respectively.
\begin{lem} \label{lem:nabla_U-omega-small}
  There exists a constant~$C$ such that
  \begin{equation}
    \norm[\big]{c\brackets{\nabla_{U}\omega}\ue}_p
    <C\,\norm{U}\,\norm{\omega}\,\epsilon^{\slfrac{r}{8}}
  \end{equation}
  for all vertical vectors~$U\in T_pM$, all $\omega \in \Lambda^2 T_pM$ being the horizontal lift of some element in the image of~$\mathcal{R}_N$, and all $\epsilon \in \brackets{0,1}$.
\end{lem}
\begin{proof}
  Fix an $\epsilon \in \brackets{0,1}$, a vertical vector $U\in T_pM$, and some $\omega\in \Lambda^2 T_pM$ the horizontal lift of $\overline{\omega}\coloneqq \overline{L} \overline{\eta}$ for some $\overline{\eta} \in \brackets[\normalsize]{\ker \overline{L}}^\perp \subset \Lambda^2 T_{f\brackets{p}}N$. We obtain, similarly as in the proof of \cref{Lem-c-barc-small}, with
  \begin{equation} 
    \omega=\sum_\alpha g_N\brackets{\overline{\eta},\overline{\omega}_\alpha} L \overline{\omega}_\alpha, \quad
    C_1\coloneqq \tbinom{n}{2} \norm[\big]{\restrict{\overline{L}}{\brackets[\normalsize]{\ker \overline{L}}^\perp}^{-1}}^2_{\text{op}}
  \end{equation}
  at the point~$p$ and a suitable constant~$C_2$ independent of~$U$,~$X$,~$\omega$ and~$\epsilon$
  \begin{equation} \label{Proof-Lem-NablaUomega-small1}
    \begin{split}
      \norm[\big]{c\brackets{\nabla_{U}\omega}\ue}_p^2
      =\norm[\Big]{\sum_\alpha g_N\brackets{\overline{\eta},\overline{\omega}_\alpha} c\brackets[\big]{\nabla_{U}\brackets[\normalsize]{L\overline{\omega}_\alpha}}\ue}_p^2
      &\leq C_1 \norm{\omega}^2 \norm[\Big]{\sum_\alpha \abs[\big]{c\brackets[\big]{\nabla_{U}\brackets[\normalsize]{L\overline{\omega}_\alpha}}\ue}_p^2}_\A\\
      &\overset{(\star)}{<} C_1\,C_2\,\norm{U}^2\norm{\omega}^2\epsilon^{\slfrac{r}{4}}.
    \end{split}
  \end{equation}
  In the first step, we used that the term on the left-hand side is tensorial in~$\omega$. It remains to show the estimate $(\star)$ in \cref{Proof-Lem-NablaUomega-small1}. All constants in the remaining proof are chosen independent of~$U$,~$\omega$ and~$\epsilon$. We start calculating
    \begin{equation} \label{EqNablaNabla}
    \begin{split}
      &\nabla^2_{U} \sqbrackets{c\brackets{L\overline{\omega}_\alpha}\!-\!\overline{c}\brackets[\normalsize]{f^*\overline{L}\overline{\omega}_\alpha}}^2\ue\\
      &\quad =2c\brackets[\big]{\nabla_{U} \brackets[\normalsize]{L\overline{\omega}_\alpha}}^2\ue+\sqbrackets{c\brackets{L\overline{\omega}_\alpha}\!-\!\overline{c}\brackets[\normalsize]{f^*\overline{L}\overline{\omega}_\alpha}}^2\nabla^2_{U}\ue\\
      &\quad \quad
      \!+\!c\brackets{\nabla^2_{U} \brackets[\normalsize]{L\overline{\omega}_\alpha}}\sqbrackets{c\brackets{L\overline{\omega}_\alpha}\!-\!\overline{c}\brackets[\normalsize]{f^*\overline{L}\overline{\omega}_\alpha}}\ue
      \!+\!\sqbrackets{c\brackets{L\overline{\omega}_\alpha}\!-\!\overline{c}\brackets[\normalsize]{f^*\overline{L}\overline{\omega}_\alpha}}c\brackets{\nabla^2_{U} \brackets[\normalsize]{L\overline{\omega}_\alpha}}\ue\\
      &\quad \quad 
      \!+\!2\sqbrackets{c\brackets[\big]{\nabla_{U} \brackets[\normalsize]{L\overline{\omega}_\alpha}}\sqbrackets{c\brackets{L\overline{\omega}_\alpha}\!-\!\overline{c}\brackets[\normalsize]{f^*\overline{L}\overline{\omega}_\alpha}}\!+\!\sqbrackets{c\brackets{L\overline{\omega}_\alpha}\!-\!\overline{c}\brackets[\normalsize]{f^*\overline{L}\overline{\omega}_\alpha}}c\brackets[\big]{\nabla_{U} \brackets[\normalsize]{L\overline{\omega}_\alpha}}}\nabla_{U}\ue.
    \end{split}
  \end{equation}
  We insert~$\ue$ into \cref{EqLichnerowiczRiemanSubmersion2}, take $\nabla^2_{U}$, and then the fiberwise $\A$-valued inner product in~$\brackets{\SpinBdl_\mathcal{L}}_p$ with~$\ue$. This yields by inserting \cref{EqNablaNabla}, rearranging the terms, taking the norm in the $\Cstar$-algebra~$\A$, and estimating the right-hand side with the triangle and the Cauchy-Schwarz inequality that
  \begin{equation} \label{Eq-nabla-U-omega-local}
    \begin{split}
      &2\norm[\Big]{\sum_\alpha \abs[\big]{c\brackets[\big]{\nabla_{U} \brackets[\normalsize]{L\overline{\omega}_\alpha}}\ue}_p^2}_{\A}\\
      &\leq \sum_\alpha \sqbrackets[\Big]{\norm[\big]{\nabla^2_{U}\ue}_p\norm[\big]{\sqbrackets{c\brackets{L\overline{\omega}_\alpha}\!-\!\overline{c}\brackets[\normalsize]{f^*\overline{L}\overline{\omega}_\alpha}}^2\ue}_p
      \!+\!2\norm[\big]{c\brackets[\big]{\nabla^2_{U} \brackets[\normalsize]{L\overline{\omega}_\alpha}}\ue}_p \norm{\sqbrackets{c\brackets{L\overline{\omega}_\alpha}\!-\!\overline{c}\brackets[\normalsize]{f^*\overline{L}\overline{\omega}_\alpha}}\ue}_p\\
      &\quad
      \!+\!2\norm{\nabla_{U}\ue}_p \norm{\brackets{c\brackets[\big]{\nabla_{U} \brackets[\normalsize]{L\overline{\omega}_\alpha}}\sqbrackets{c\brackets{L\overline{\omega}_\alpha}\!-\!\overline{c}\brackets[\normalsize]{f^*\overline{L}\overline{\omega}_\alpha}}\!+\!\sqbrackets{c\brackets{L\overline{\omega}_\alpha}\!-\!\overline{c}\brackets[\normalsize]{f^*\overline{L}\overline{\omega}_\alpha}}c\brackets[\big]{\nabla_{U} \brackets[\normalsize]{L\overline{\omega}_\alpha}}}\ue}_p}\\[0.5em]
      & \quad
      +\norm[\big]{\nabla^2_{U}\DL^2\ue}_p \norm{\ue}_p
      +\norm[\big]{\nabla^2_{U} \nabla^*\nabla \ue}_p \norm{\ue}_p.
    \end{split}
  \end{equation}  
  Here we used that~$c$ and~$\overline{c}$ are skew-adjoint on two vectors. By \cref{PropLinftyAlmost}~(\ref{PropLinftyAlmost1}) and~(\ref{PropLinftyAlmost2}), there exist constants~$C_3$,~$C_4$ and~$C_5$ such that $\norm{\ue}_p\leq C_3$, 
  \begin{equation}
    \begin{split}
      \norm{\nabla^2_{U}\DL^2\ue}^2_p
        \leq \brackets{m-n}^2\norm{U}^4\norm[\big]{\nabla^2 \DL^2\ue}^2_p
        &\leq C_4 \norm{U}^4\norm[\big]{\nabla^2 \DL^2\ue}_{\Ltwoblanc}^{2r} \\
        &\leq C_4 \norm{U}^4\norm[\big]{\Ltwoscalarproduct[\big]{\DL^2\ue}{\brackets{\nabla^*}^2 \DL^2 \ue}}_\A^r\\
        &\leq C_4 \norm{U}^4\Ltwonorm{\DL^2\ue}^r\norm[\big]{\brackets{\nabla^*} ^2\DL^2\ue}_{\Ltwoblanc}^r \\
        &<C_4 \epsilon^{\slfrac{r}{2}}
    \end{split}
  \end{equation}
  and similarly
  \begin{equation}
    \begin{split}
      \norm[\big]{\nabla^2_{U}\nabla^*\nabla \ue}^2_p
        \leq \brackets{m-n}^2 \norm{U}^4 \norm[\big]{\nabla^2 \nabla^*\nabla \ue}^2_p
        &\leq C_5^2 \norm{U}^4 \Ltwonorm{\nabla\ue}^r \norm[\big]{\brackets{\nabla^*}^2\nabla^*\nabla\ue}_{\Ltwoblanc}^r\\
        &<C_5 \epsilon^{\slfrac{r}{2}}.
    \end{split}
  \end{equation}
  Together with \cref{Lem-c-barc-small} and \cref{PropLinftyAlmost}~(\ref{PropLinftyAlmost3}), we obtain for a constant~$C_2$
  \begin{equation}
    \norm[\Big]{\sum_\alpha \abs[\big]{\nabla_{U} \brackets[\normalsize]{L\overline{\omega}_\alpha}\ue}_p^2}_\A<C_2\,\norm{U}^2\,\epsilon^{\slfrac{r}{4}},
  \end{equation}
  hence estimate~$(\star)$ in \cref{Proof-Lem-NablaUomega-small1} is proved.
\end{proof}
\begin{lem} \label{lem:nabla_X-omega-small}
  There exists a constant~$C$ and a linear map $B\colon T_pM\to\End(\Lambda^2 \mathcal{H}_p)$ such that
  \begin{equation}
    \norm[\big]{\sqbrackets[\big]{c\brackets{\mathcal V\nabla_X\omega}+c\brackets[\big]{B_X \omega}-\overline{c}\brackets{f^*\overline{B_X \omega}}}\ue}_p< C\,\norm{X}\,\norm{\omega}\,\epsilon^{\slfrac{r}{8}}
  \end{equation}
  for every horizontal vector~$X\in T_pM$, every~$\omega \in \Lambda^2 T_pM$ being the horizontal lift of an element in the image of the curvature operator of~$N$, and for all $\epsilon \in \brackets{0,1}$. Here $\overline{B_X \omega}$ denotes the projection of $B_X \omega$ to $\Lambda^2 T_{f\brackets{p}}N$ by the map~$\Lambda^2 \mathd f$.
\end{lem}
\begin{proof}
  Fix an $\epsilon\in\brackets{0,1}$, some horizontal vector~$X\in T_pM$ and some~$\omega$ the horizontal lift of some $\overline{\omega}\in \image \mathcal{R}_N$. We define~$B_X \omega\in \Lambda^2 T_pM$ as the horizontal lift of 
  \begin{equation}
    \brackets{\nabla_{\mathd f\brackets{X}}\overline{L}}\overline{L}^{-1}\overline{\omega} \in \Lambda^2 T_{f\brackets{p}}N,
  \end{equation}
  where $\overline{L}^{-1}$ is the linear map given by the inverse of
  $\restrict{\overline{L}}{\brackets[\normalsize]{\ker \overline{L}}^\perp}\colon \brackets[\normalsize]{\ker \overline{L}}^\perp \to \image \overline{L}$
  extended by zero. Pick some $\overline{\eta}\in \brackets[\normalsize]{\ker \overline{L}}^\perp$ such that $\overline{\omega}=\overline{L}\overline{\eta}$, define $\overline{X}\coloneqq \mathd f\brackets{X}$, and denote by~$\widehat{}$\, the horizontal lift of a 2-vector of~$N$. We obtain as in the proof of \cref{Lem-c-barc-small} for a suitable constant~$C_1$
  \begin{equation} \label{eq:nabla_X-proof-eq1}
    \begin{split}
    &\norm[\big]{\sqbrackets[\big]{c\brackets{\mathcal{V}\nabla_X\omega}+c\brackets[\big]{B_X \omega}-\overline{c}\brackets{f^*\overline{B_X \omega}}}\ue}_p^2\\
    &\quad \quad = \norm[\Big]{\sum_\alpha g\brackets{\overline{\eta},\overline{\omega}_\alpha}
      \sqbrackets[\big]{
        c\brackets[\big]{\mathcal{V}\nabla_X\brackets{L\overline{\omega}_\alpha}}
        +c\brackets[\big]{\widehat{\brackets[\normalsize]{\nabla_{\bar{X}}\overline{L}}\,\overline{\omega}_\alpha}}
        -\overline{c}\brackets{f^*\brackets[\normalsize]{\nabla_{\bar{X}}\overline{L}}\,\overline{\omega}_\alpha}
        }\ue
      }_p^2\\
    &\quad \quad \leq C_1 \Anorm[\Big]{\sum_\alpha \abs[\big]{
      \sqbrackets[\big]{
        c\brackets[\big]{\mathcal{V}\nabla_X\brackets{L\overline{\omega}_\alpha}}
        +c\brackets[\big]{\widehat{\brackets[\normalsize]{\nabla_{\bar{X}}\overline{L}}\,\overline{\omega}_\alpha}}
        -\overline{c}\brackets{f^*\brackets[\normalsize]{\nabla_{\bar{X}}\overline{L}}\,\overline{\omega}_\alpha}
        }\ue
      }_p^2}
    \end{split}
  \end{equation}
  Here we used that the left-hand side of \cref{eq:nabla_X-proof-eq1} is tensorial in~$\omega$, and that~$B_X\omega$ is equal to the horizontal lift of $\brackets{\nabla_{\bar{X}}\overline{L}}\overline{L}^{-1}\overline{L}\overline{\eta}=\brackets[\normalize]{\nabla_{\bar{X}}\overline{L}}\overline{\eta}$. Note that 
  \begin{equation} \label{eq:nabla_X-proof-eq2}
    \begin{split}
      &\nabla_X\brackets{c\brackets{L\overline{\omega}_\alpha}-\overline{c}\brackets[\normalize]{f^*\overline{L}\overline{\omega}_\alpha}} \\
      &\quad = c\brackets{\mathcal{V}\nabla_X\brackets{L\overline{\omega}_\alpha}}+c\brackets{\mathcal{H}\nabla_X\brackets{L\overline{\omega}_\alpha}}-\overline{c}\brackets{f^*\nabla_{\bar{X}}\brackets[normalsize]{\overline{L}\overline{\omega}_\alpha}} \\
      & \quad = c\brackets{\mathcal{V}\nabla_X\brackets{L\overline{\omega}_\alpha}}
      +c\brackets[\big]{\widehat{\brackets[normalsize]{\nabla_{\bar{X}}\overline{L}}\overline{\omega}_\alpha}}-\overline{c}\brackets{f^*\brackets[normalsize]{\nabla_{\bar{X}}{\overline{L}}}\overline{\omega}_\alpha}
      + c\brackets[\big]{\widehat{\overline{L}\nabla_{\bar{X}}\overline{\omega}_\alpha}}-\overline{c}\brackets{f^*\brackets[normalsize]{\nabla_{\bar{X}}\overline{L}}\overline{\omega}_\alpha}.
    \end{split}
  \end{equation}
  The desired estimate follows from \cref{eq:nabla_X-proof-eq1} and \cref{eq:nabla_X-proof-eq2} as in the proof of \cref{lem:nabla_U-omega-small}, namely by taking two times the covariant derivative of \cref{EqLichnerowiczRiemanSubmersion2} and using \cref{PropLinftyAlmost} and \cref{Lem-c-barc-small}.
\end{proof}
\section{Pointwise computations} \label{sec:pointwise-computations}
Let~$f\colon M\to N$ be as in \cref{MainThm}. By \cref{ThmRiemSubmersion}, the map~$f$ is a Riemannian submersion. In this section we will begin working with the pointwise inequalities derived in \cref{subsec:establishing-pointwise-ineq}, connecting them to local geometric quantities such as the O'Neill tensors $A, T$ and the curvatures of $M, N$. \cref{lem:family-ue-summary} below summarises these pointwise relations and brings them into a more digestible form. On first reading one may treat \cref{lem:family-ue-summary} as a black box.
\medbreak

In \cref{subsec:ricc-curv-in-riem-submersion} we prove the relation $\Ric_M=f^*\Ric_N$. \cref{subsec:pointwise-eq-for-A-and-T} expresses our pointwise equations with the tensors $A, T$ and upgrades these by applying chains $\ad_{\omega_i}$ to them, see \cref{prop: pointwise-TA}. In order to work with these relations we need to know something about what the image of the curvature operator $\mathcal R_N$ looks like. \cref{subsec: curvature-image} is then an abstract analysis of $\image\mathcal R_N$. \cref{sec: curv-cliff} applies this analysis to the relations of \cref{subsec:pointwise-eq-for-A-and-T}, allowing us to conclude that the fibres of $f$ have vanishing mean curvature (cf. \cref{prop: real-case,prop: complex-case}).
\medbreak

We begin by summarsing the results of \cref{sec:setup-pointwise-eq} in \cref{lem:family-ue-summary} below. Let~$\mathcal{H}$ and~$\mathcal{V}$ denote the horizontal and vertical subbundles of~$TM$, respectively. When there is no risk of confusion, the same symbols also denote the projections onto these distributions. We use the following short notation:
\begin{defi}
  For $p\in M$ and $\omega\in \Lambda^2 T_pM$ the horizontal lift of some $\overline{\omega}\in \Lambda^2 T_{f\brackets{p}}N$, we define 
  $$(c-\overline c)(\omega)\coloneq c(\omega)-\overline c\brackets{f^*\overline{\omega}}.$$ 
\end{defi}
\begin{lem} \label{lem:family-ue-summary}
  For every point~$p\in M$, there exists a family $\set{\ue}\subset \brackets{\SpinBdl_\mathcal{L}}_p\cong \ComplexCl_{m,n}\otimes \Cstar \pi_1\brackets{M}$ with $\norm{\ue}_p=1$ together with constants~$C,r$ and a linear map $B\colon T_pM\to\End(\Lambda^2 \mathcal{H}_p)$ such that
  \begin{enumerate}
    \item \label{item1:lem-family-ue-summary}
      $\norm[\big]{R^{\SpinBdl_\mathcal{L}}_{Y,Z} \ue}_p\leq C\,\norm{Y}\,\norm{Z}\,\epsilon^r$,
    \item \label{item2:lem-family-ue-summary}
      $\norm{(c-\overline c)\,(\omega)\ue}_p\leq C\,\norm{\omega}\, \epsilon^r$,
    \item \label{item3:lem-family-ue-summary}
      $\norm{c\brackets{\nabla_U \omega}\ue}_p\leq C\,\norm{U}\,\norm{\omega}\,\epsilon^r$, and
    \item \label{item4:lem-family-ue-summary}
      $\norm{c\brackets{\mathcal V\nabla_X\omega}\ue+(c-\overline c)\brackets{B_X\omega}\ue}_p\leq C\,\norm{X}\,\norm{\omega}\,\epsilon^r$
  \end{enumerate}
  for all~$\epsilon>0$, all $Y,Z\in T_pM$, all vertical $U\in T_pM$, all horizontal $X\in T_pM$, and all $\omega\in \Lambda^2 T_p M$ being the horizontal lift of some element in the image of the curvature operator of~$N$.
\end{lem}
\begin{rem}
Even though derivatives appear, parts~(\ref{item3:lem-family-ue-summary}) and~(\ref{item4:lem-family-ue-summary}) of \cref{lem:family-ue-summary} are tensorial in $\omega$. Furthermore, the precise form of $B$ will be irrelevant for our considerations.
\end{rem}
\begin{proof}[Proof of \cref{lem:family-ue-summary}]
  Fix some point $p\in M$. By \cref{LemExistenceAlmostHarmonicSections} there exists a family of almost~$\DL$-harmonic section. By rescaling appropriately, which is possible by \cref{PropLinftyAlmost}~(\ref{PropLinftyAlmost4}), and restricting to the point~$p$, we obtain a family~$\ue$ which satisfies by \cref{Lem-CurvatureOperator-small}, \ref{Lem-c-barc-small}, \ref{lem:nabla_U-omega-small} and~\ref{lem:nabla_X-omega-small} for appropriate constants~$C,r$ the estimates (\ref{item1:lem-family-ue-summary})-(\ref{item4:lem-family-ue-summary}) for all~$\epsilon\in \brackets{0, \epsilon_0}$ for some~$\epsilon_0\in \brackets{0,1}$. For~$\epsilon\geq \epsilon_0$, we just redefine $u_\epsilon\coloneqq u_{\slfrac{\epsilon_0}{2}}$ and the lemma is proved. 
\end{proof}
\subsection{The Ricci curvature in the Riemannian submersion} \label{subsec:ricc-curv-in-riem-submersion}
The scalar curvature of~$M$ is by \cref{ThmRiemSubmersion} equal to the scalar curvature of~$N$. In this section we show that under the assumption of \cref{MainThm} this is even true for the Ricci curvatures of the manifolds~$M$ and~$N$, to be precise the following holds:
\begin{prop} \label{PropRicUp=RicDown}
  The Ricci curvature of~$M$ and~$N$ are related via $\ric_M=\ric_N\circ \brackets{\mathd f\otimes \mathd f}$.
\end{prop}
\begin{proof}
  We fix a point $p\in M$ and a tangent vector $X\in T_pM$. Let~$\set{\ue}\subset \brackets{\SpinBdl_\mathcal{L}}_p$ be a family as in \cref{lem:family-ue-summary}, and denote $\overline{X}\coloneqq \mathd f\brackets{X}$. The Ricci curvature on~$M$ and the curvature tensor~$R^{\SpinBdl M}$ on the spinor bundle~$\SpinBdl M$ are related via
  \begin{equation} \label{eq-ricci-curvaturetensor}
    c\brackets{\Ric_M\brackets{X}}  = -2\sum_{\gamma=1}^m c\brackets{e_\gamma}R^{\SpinBdl M}_{X,e_\gamma}.
  \end{equation}
  Since the Mishchenko--Fomenko bundle~$\mathcal{L}$ is flat and the Dirac bundle $\SpinBdl_\mathcal{L}=\SpinBdl\otimes \mathcal{L}$ is locally equal to the graded tensor product of $\SpinBdl M$ and $\SpinBdl N$, the curvature tensor on $\SpinBdl_\mathcal{L}$ splits, and we obtain
  \begin{equation} \label{eq-2-proof-Ricci}
    \begin{split}
      2\sum_{\gamma=1}^m c\brackets{e_\gamma} R^{\SpinBdl_\mathcal{L}}_{e_\gamma, X}\ue
      &=  2\sum_{\gamma=1}^m c\brackets{e_\gamma} R^{\SpinBdl M}_{e_\gamma, X}\ue+
          2\sum_{\gamma=1}^n c\brackets{e_\gamma} R^{\SpinBdl N}_{\bar{e}_\gamma, \bar{X}}\ue.
    \end{split}
  \end{equation}
  The first term on the right-hand side is equal to $c\brackets{\Ric_M\brackets{X}} \ue$ by \cref{eq-ricci-curvaturetensor}. We rewrite the second term using the formula for the curvature tensor of the spinor bundle \cite[Chapter II Theorem 4.15]{Lawson1989}. Note that we have a different sign, since we consider the spinor bundle associated to~$TN$ equipped with the negative definite metric~$\brackets{-g_N}$ (see \cref{ExampleSpinMaps}). We obtain
  \begin{equation} \label{eq-3-proof-Ricci}
    \begin{split}
      2\sum_{\gamma=1}^n c\brackets{e_\gamma} R^{\SpinBdl N}_{\bar{e}_\gamma, \bar{X}}\ue
      &=-2\sum_{\gamma=1}^n c\brackets{e_\gamma} \sum_{k<\ell}g_N\brackets[\big]{R^{TN}_{\bar{e}_\gamma, \bar{X}}\overline{e}_k,\overline{e}_\ell}\overline{c}\brackets{f^*\overline{e}_k}\overline{c}\brackets{f^*\overline{e}_\ell}\ue \\
      &=-2\sum_{\gamma=1}^n c\brackets{e_\gamma} \sum_{k<\ell}g_N\brackets[\big]{R^{TN}_{\bar{e}_\gamma, \bar{X}}\overline{e}_k,\overline{e}_\ell}c\brackets{e_k}c\brackets{e_\ell}\ue \\
      &\quad +2\sum_{\gamma=1}^n c\brackets{e_\gamma} \sum_{k<\ell}g_N\brackets[\big]{R^{TN}_{\bar{e}_\gamma, \bar{X}}\overline{e}_k,\overline{e}_\ell}\brackets{c\brackets{e_k}c\brackets{e_\ell}-\overline{c}\brackets{f^*\overline{e}_k}\overline{c}\brackets{f^*\overline{e}_\ell}}\ue.
    \end{split}
  \end{equation}
  All summand in the first term in \cref{eq-3-proof-Ricci} with three pairwise disjoint indices vanishes by the Bianchi identity, hence this term is equal to $c\brackets[\big]{\widehat{\Ric}_N\brackets[\normalsize]{\overline{X}}}\ue$ by a short calculation analyzing the summands with $\gamma=k$ and $\gamma=\ell$. Here $\widehat{\Ric}_N\brackets[\normalsize]{\overline{X}}$ denotes the horizontal lift of $\Ric_N\brackets[\normalsize]{\overline{X}}$. Putting together \cref{eq-ricci-curvaturetensor}, (\ref{eq-2-proof-Ricci}) and (\ref{eq-3-proof-Ricci}) yields
  \begin{equation} \label{eq-4-proof-Ricci}
    \begin{split}
      &c\brackets[\big]{\Ric_M\brackets{X}-\widehat{\Ric}_N\brackets[\normalsize]{\overline{X}}}\ue \\
      &\quad=-2\sum_{\gamma=1}^m c\brackets{e_\gamma} R^{\SpinBdl_\mathcal{L}}_{e_\gamma, X}\ue
      +2\sum_{\gamma=1}^n c\brackets{e_\gamma} \sum_{k<\ell}g_N\brackets[\big]{R^{TN}_{\bar{e}_\gamma, \bar{X}}\overline{e}_k,\overline{e}_\ell}\brackets{c\brackets{e_k}c\brackets{e_\ell}-\overline{c}\brackets{f^*\overline{e}_k}\overline{c}\brackets{f^*\overline{e}_\ell}}\ue.
    \end{split}
  \end{equation}
  Note that 
  \begin{equation} \label{eq-5-proof-Ricci}
    \mathcal{R}_N\brackets[\normalsize]{\overline{X}\wedge \overline{e}_\gamma}
    =\sum_{k<\ell} g_N\brackets[\normalsize]{\mathcal{R}_N\brackets{\overline{X}\wedge \overline{e}_\gamma},\overline{e}_k\wedge \overline{e}_\ell} \overline{e}_k\wedge \overline{e}_\ell
    =\sum_{k<\ell} g_N\brackets[\big]{R^{TN}_{\overline{e}_\gamma,\overline{X}} \overline{e}_k, \overline{e}_\ell} \overline{e}_k\wedge \overline{e}_\ell,
  \end{equation}
  hence the second term on the right-hand side in \cref{eq-4-proof-Ricci} is equal to
  \begin{equation} \label{eq-6-proof-Ricci}
    2\sum_{\gamma=1}^n c\brackets{e_\gamma} \brackets{c-\overline{c}}\brackets[\big]{\widehat{\mathcal{R}_N\brackets[\normalsize]{\overline{X}\wedge \overline{e}_\gamma}}}\ue.
  \end{equation}
  Here $\widehat{\mathcal{R}_N\brackets[\normalsize]{\overline{X}\wedge \overline{e}_\gamma}}$ denotes the horizontal lift of $\mathcal{R}_N\brackets[\normalsize]{\overline{X}\wedge \overline{e}_\gamma}$. We insert \cref{eq-6-proof-Ricci} into \cref{eq-4-proof-Ricci}, take $\norm{\placeholder}_p$ of both sides, and obtain
  \begin{equation} \label{eq-67-proof-Ricci}
    \begin{split}
      \norm[\big]{c\brackets[\big]{\Ric_M\brackets{X}-\widehat{\Ric}_N\brackets[\normalsize]{\overline{X}}}\ue}
      &\leq 2\sum_{\gamma=1}^m \norm[\big]{R^{\SpinBdl_\mathcal{L}}_{e_\gamma, X}\ue}
      +2\sum_{\gamma=1}^n \norm[\big]{\brackets{c-\overline{c}}\brackets[\big]{\widehat{\mathcal{R}_N\brackets[\normalsize]{\overline{X}\wedge \overline{e}_\gamma}}}\ue} \\
      &< C\,\norm{X}\,\epsilon^r
    \end{split}
  \end{equation}
  for suitable constants~$C,r$ which are independent of~$\epsilon$. Here we used in the first step the triangle inequality, that $c\brackets{e_\gamma}$ is skew-adjoint, and $c\brackets{e_\gamma}c\brackets{e_\gamma}=-1$. The second step holds by \cref{lem:family-ue-summary}~(\ref{item1:lem-family-ue-summary}) and~(\ref{item2:lem-family-ue-summary}). Since $\norm{\ue}_p=1$, we obtain 
  \begin{equation}
    \norm[\big]{\Ric_M\brackets{X}-\widehat{\Ric}_N\brackets[\normalsize]{\overline{X}}}<C\,\norm{X}\,\epsilon^r,
  \end{equation}
  and $\Ric_M\brackets{X}=\widehat{\Ric}_N\brackets[\normalsize]{\overline{X}}$ in the limit $\epsilon \to 0$.
\end{proof}

\subsection{Pointwise equation for the O'Neill tensors $A$ and $T$} \label{subsec:pointwise-eq-for-A-and-T}

In this section we derive a system of pointwise equations for the O'Neill tensors $A, T$ of the Riemannian submersion $f:(M,g_M)\to(N,g_N)$. These relations show that certain $2$-vectors fail to be invertible in the Clifford algebra. These $2$-vectors are built from the tensors $A, T$, the elements of $\image(\mathcal R_N)_{f(p)}$, and the tensor $B$ from \cref{lem:family-ue-summary}. The main result of this section is the following:

\begin{prop}\label{prop: pointwise-TA}
Let $p\in M$ and let $\ue$, $r$, $B$ be as in \cref{lem:family-ue-summary}. Then for each $\ell\in\N$ there exists a constant~$C_\ell$ so that:
\begin{enumerate}
\item\label{item: TA1} Let $\overline\omega$ be an element of the Lie-algebra generated by $\image(\mathcal R_N)_{f(p)}$ and let $\omega$ denote its horizontal lift to $\Lambda^2T_pM$. For all $\epsilon>0$
\begin{equation}
\|(c-\overline c)\,(\omega)\ue\|\leq  C_1\,\|\omega\|\epsilon^r.\label{eq: c-cbar-algebra}
\end{equation}
\item\label{item: TA2} Let $U\in T_pM$ be a vertical vector, and $\omega_1,...,\omega_\ell\in\Lambda^2 T_pM$ horizontal lifts of elements of $\image(\mathcal R_N)_{f(p)}$. For all $\epsilon>0$
\begin{equation}
 \|c((\ad_{\omega_1}\circ\ldots\circ\ad_{\omega_\ell})(2T_U+A_U^*))\ue\|\leq C_\ell\,\|\omega_1\|\cdots\|\omega_\ell\|\,\|U\|\epsilon^r.\label{eq: T-B-chain}
\end{equation}
\item\label{item: TA3} Let $X\in T_pM$ be horizontal, and $\omega_1,...,\omega_\ell\in\Lambda^2 T_pM$ horizontal lifts of elements of $\image(\mathcal R_N)_{f(p)}$. For all $\epsilon>0$
\begin{equation}\label{eq: A-chain}
\begin{split}
&\norm{c\brackets{(\ad_{\omega_1}\circ\ldots\circ\ad_{\omega_\ell})(2A_X)}\ue+(c-\overline c)\brackets{(\ad_{\omega_1}\circ\ldots\circ\ad_{\omega_{\ell-1}})(B_X\omega_\ell)}\ue}\\
&\qquad\qquad\leq C_\ell\|\omega_1\|\cdots \|\omega_\ell\|\,\|X\|\,\epsilon^r.\end{split}
\end{equation}\end{enumerate}
\end{prop}
For the convenience of the reader we recall several of the definitions and objects appearing in \cref{prop: pointwise-TA}.
\begin{rem}\begin{enumerate}
\item For a linear map $L:T_pM\to T_pM$ we had defined
$$c(L)=\sum_{ij} g_M(Le_i,e_j)\,c(e_i)c(e_j),$$
where $e_1,...,e_m$ is some orthonormal basis of $T_pM$.
\item We recall the definitions of $A, A^*, T^*$. Let $p\in M$ a point, $U,W\in\mathcal V_p$ vertical and $X,Y\in\mathcal H_p$ horizontal vectors. We had defined the tensors $T^*_U:\mathcal H_p\to\mathcal V_p$, $A_X:\mathcal H_p\to\mathcal V_p$, $A^*_U:\mathcal H_p\to\mathcal H_p$, in \cref{def: A-T-variants} via the equations
$$g_M(T_U^*X, W) = -g_M(\nabla_UX, W),\quad g_M(A_XY,U)= \tfrac12\, g_M(U,[X,Y])= g_M(A_U^*X,Y).$$
These identities do not depend on the choice of extensions of $W,X,Y$ to vector fields.
\item $B$ is a linear map $T_pM\to\End(\Lambda^2\mathcal H_p)$. Its precise form is irrelevant.
\item For $U, X,\omega$ as in the proposition we extend $A_X, A_U^*,T_U^*, B_X(\omega)$ to linear maps $T_pM\to T_pM$ by letting them act as zero on the orthogonal complements of their domains of definition.
\item If $\omega\in\Lambda^2T_pM$ and $L\in\End(T_pM)$ we use $\ad_{\omega}(L)$ to denote the commutator $[\omega,L]=\omega\cdot L -L\cdot \omega$, where we identify $\Lambda^2 T_pM$ with $\mathfrak{so}(T_pM)\subseteq\End(T_pM)$ as in \cref{Nota-frames-ClMult.byEnd-L} point (\ref{item2:Nota-frames-ClMult.byEnd-L}).
\item Since $T_U^*$ maps the horizontal space $\mathcal H_p$ to the vertical space $\mathcal V_p$ and vanishes on $\mathcal V_p$, one has that
$$(\ad_{\omega_1}\circ\ldots\ad_{\omega_{\ell}})(T_U)=(-1)^\ell T_U\cdot \omega_\ell\cdots \omega_1=T_U\cdot(\omega_1\cdots\omega_\ell)^T,$$
where $\cdot$ denotes the product in $\End(T_pM)$ and $^T$ the transpose. Similarly for $A_X$. On the other hand $A_U^*$ and $B_X(\omega_\ell)$ are anti-symmetric linear maps $\mathcal H_p$ to $\mathcal H_p$ and so the terms involving their commutators do not simplify.
\end{enumerate}
\end{rem}

In order to prove \cref{prop: pointwise-TA} we will make use of the following elementary lemma:
\begin{lem}\label{lemma: clifford-to-matrix}
Let $(V,\scalarproduct{}{})$ be a real Hilbert space, $L_1, L_2\in\End(V)$. Then
\begin{align}
[c(L_1),c(L_2)]&=-2\,c(L_1\cdot L_2)+2\,c(L_1\cdot L_2^T) - 2\,c(L_2^T\cdot L_1)+2\,c(L_2\cdot L_1).\\
[\overline c(L_1),\overline c(L_2)]&=+2\,\overline c(L_1\cdot L_2)-2\,\overline c(L_1\cdot L_2^T) + 2\,\overline c(L_2^T\cdot L_1)-2\,\overline c(L_2\cdot L_1).
\end{align}
Here $c, \overline c$ are the Clifford multiplications satisfying $c(v)^2 = -\scalarproduct{v}{v}$ and $\overline c(v)^2=+\scalarproduct{v}{v}$ for any $v\in V$. $[\,,\,]$ denotes the commutator in the corresponding Clifford algebra.
\end{lem}
\begin{proof}
Let $e_1,...,e_n$ be an orthonormal basis of $V$. Then for $i,j,k,\ell\in\{1,...,n\}$ one has:
\begin{align}
c(e_i)c(e_j)c(e_k)c( e_\ell) =& - 2 \delta_{jk} c(e_i) c(e_k) - c(e_i)c(e_k)c(e_j)c(e_\ell)\\
=&- 2 \delta_{jk} c(e_i) c(e_\ell) +2\delta_{j\ell}c(e_i)c(e_k)+ c(e_i)c(e_k)c(e_\ell)c(e_j)\\
=& -2\delta_{jk}c(e_i)c(e_\ell)+2\delta_{j\ell}c(e_i)c(e_k)-2\delta_{ik}c(e_\ell)c(e_j)- c(e_k)c(e_i)c(e_\ell)c(e_j)\\
\begin{split}
=&-2\delta_{jk}c(e_i)c(e_\ell)+2\delta_{j\ell}c(e_i)c(e_k)-2\delta_{ik}c(e_\ell)c(e_j)+2\delta_{i\ell}c(e_k)c(e_j)\\
&\quad+c(e_k)c(e_\ell)c(e_i)c(e_j).
\end{split}
\end{align}
By noting that
$$c(L_1)c(L_2) - c(L_2)c(L_1)=\sum_{ijk\ell} (L_1)_{ij}(L_2)_{k\ell} \brackets{c(e_i)c(e_j)c(e_k)c( e_\ell)  - c(e_k)c(e_\ell)c(e_i)c(e_j)}$$
and dutifully performing all contractions the equation for $[c(L_1), c(L_2)]$ follows. The proof for the equation with $\overline c$ is identical, note that the signs of all contractions are flipped.\end{proof}
\begin{rem}\label{rem: clifford-to-matrix}
Suppose one of $L_1, L_2$ is anti-symmetric, then \cref{lemma: clifford-to-matrix} states
$$[c(L_1),c(L_2)]=-4 c(\ad_{L_1}L_2),\qquad [\overline c(L_1), \overline c(L_2)]=4\overline c(\ad_{L_1}L_2).$$
\end{rem}
The key computation to prove \cref{prop: pointwise-TA} is then the following:

\begin{lem}\label{lemma: chain}
Let $p\in M$ and $\ue$ be as in \cref{lem:family-ue-summary}, suppose $\eta_1\in\End(T_pM)$ and $\overline\eta_2\in\End(f^*(TN)_p)$ so that
\begin{equation}
\|c(\eta_1)\ue+\overline c(\overline\eta_2)\ue\|\leq C\,(\|\eta_1\|+\|\overline\eta_2\|)\,\epsilon^r\label{eq: chain-start}
\end{equation}
for all $\epsilon>0$. Then there exist constants $C_\ell'$ for $\ell\in N$ so that 
$$\|c((\ad_{\omega_1}\circ \ldots \circ\ad_{\omega_\ell})\,(\eta_1))\ue+\overline{c}((\ad_{\overline\omega_1}\circ\ldots\circ\ad_{\overline\omega_\ell})(\overline\eta_2))\ue\|\leq C_\ell' \|\omega_1\|\cdots \|\omega_\ell\|\,(\|\eta_1\|+\|\overline\eta_2\|)\,\epsilon^r$$
for all $\omega_1,...,\omega_\ell\in\Lambda^2T_pM$ being horizontal lifts of $\overline\omega_1,...,\overline\omega_\ell\in\image(\mathcal R_N)_{f(p)}$.
\end{lem}
\begin{proof}
We show the claim for $\ell=1$, the general case then follows by iteratively applying the case $\ell=1$ to
$$\eta_1'=(\ad_{\omega_2}\circ\ldots\circ\ad_{\omega_\ell})(\eta_1),\qquad \overline\eta_2'=(\ad_{\overline\omega_2}\circ\ldots\circ\ad_{\overline\omega_\ell})(\overline\eta_2).$$
To this end let $\omega$ be a horizontal lift of an element $\overline\omega\in\image(\mathcal R_N)_{f(p)}$. Note the smallness of the commutator
\begin{align}
\left\|\big[(c\!-\!\overline c)(\omega), c(\eta_1)\!+\!\overline c(\overline\eta_2)\big]\ue\right\|&
\leq \|(c\!-\!\overline c)(\omega)\|\,\|(c(\eta_1)\!+\!\overline c(\overline\eta_2))\ue\|\!+\!\|c(\eta_1)\!+\!\overline c(\overline\eta_2)\|\,\|(c\!-\!\overline c)(\omega)\ue\|\\
& \leq C'\,\|\omega\|\,(\|\eta_1\|+\|\eta_2\|)\,\epsilon^r
\end{align}
for some constant $C'$. Here we used point~(\ref{item2:lem-family-ue-summary}) of \cref{lem:family-ue-summary} and \cref{eq: chain-start}. Recalling \cref{lemma: clifford-to-matrix}, in particular \cref{rem: clifford-to-matrix}, one sees that
$$[(c-\overline c)(\omega), c(\eta_1)+\overline c(f^*\overline\eta_2)]=-4\, c(\ad_\omega\eta_1)-4\,\overline c(f^*\ad_{\overline\omega}\overline\eta_2),$$
where we used additionally $[c(L_1),\overline c(\overline L_2)]=0$ for any $L_1\in\End(T_pM)$, $\overline L_2\in\End(f^*(TN)_p)$, this commutator being of two even elements of sub-algebras that graded commute with each other. This completes the proof.
\end{proof}

\begin{proof}[Proof of \cref{prop: pointwise-TA}]
We first show part~(\ref{item: TA1}) of \cref{prop: pointwise-TA}. Recall that
$$\|(c-\overline c)\,(\omega)\ue\|\leq C\,\|\omega\|\,\epsilon^r$$
for all $\omega$ being horizontal lifts of some $\overline\omega\in\image(\mathcal R_N)_{f(p)}$. We apply \cref{lemma: chain} with $\eta_1=\omega$, $\overline\eta_2=\overline\omega$ to get
$$\|(c-\overline c)\,((\ad_{\omega_1}\circ\ldots\circ\ad_{\omega_\ell})(\omega))\ue\|\leq C_\ell' \|\omega_1\|\cdots \|\omega_\ell\|\,\|\omega\|\,\epsilon^r.$$
Now any element of the Lie-algebra generated by $\image(\mathcal R_N)_{f(p)}$ is by definition a linear combination of elements $(\ad_{\overline\omega_1}\circ\ldots\circ\ad_{\overline\omega_\ell})(\overline\omega)$. Fix some basis $\overline\omega_\alpha$ of this Lie-algebra, let $\omega_\alpha$ denote the horizontal lift of the basis elements. Choosing some way of writing this basis as a linear combination of $\ad$-chains gives a constant~$C'$ so that 
$$\|(c-\overline c)(\omega_\alpha)\ue\|\leq C'\,\epsilon^r,$$
for all $\alpha$. Taking linear combinations of the $\omega_\alpha$ then gives part~(\ref{item: TA1}) for an appropriate constant~$C_1$.
\medbreak

For part~(\ref{item: TA2}) note that from \cref{lem:family-ue-summary} part~(\ref{item3:lem-family-ue-summary}) and \cref{lemma: derivative-2-form} one has
$$\|c(\ad_{\omega_\ell}(2T_U^*+A_U^*))\ue\|=\|c(\nabla_U \omega_\ell)\ue\|\leq C\,\|U\|\,\|\omega_\ell\|\,\epsilon^r$$
for all $U\in T_pM$ vertical and $\omega_\ell$ being a horizontal lift of an element of $\image(\mathcal R_N)_{f(p)}$. Part~(\ref{item: TA2}) of the proposition then follows immediately from \cref{lemma: chain}. The proof for part~(\ref{item: TA3}) is identical: from \cref{lem:family-ue-summary} part~(\ref{item4:lem-family-ue-summary}) and \cref{lemma: derivative-2-form} one has
$$\|c(2\ad_{\omega}A_X)\ue+(c-\overline c)(B_X\omega)\ue\|=\|c(\mathcal V\nabla_X\omega)\ue+(c-\overline c)(B_X\omega)\ue\|\leq C\,\|X\|\,\|\omega\|\,\epsilon^r$$
for all $X\in T_pM$ horizontal and $\omega$ a horizontal lift of an element of $\image(\mathcal R_N)_{f(p)}$. Now one applies \cref{lemma: chain}.
\end{proof}

\begin{rem}
We have then seen that $(c-\overline c)(\omega)\ue$ is small also for horizontal lifts of elements of the Lie-algebra generated by the image of $\mathcal R_N$, not just for the image alone. Of course one can now go through all previous computations and verify that any conclusion we had drawn for $\image\mathcal R_N$ then also extends to the Lie-algebra generated by this image, in particular \cref{lem:family-ue-summary} and \cref{prop: pointwise-TA}. In our attempt to write a legible proof we have opted not to do this.

Note however that the statement of part~(\ref{item: TA2}) of \cref{prop: pointwise-TA} can upgrade itself to the Lie-algebra, since $[\ad_{\omega_1},\ad_{\omega_2}]=\ad_{[\omega_1,\omega_2]}$. The only consequence of not having part~(\ref{item: TA3}) formulated for the Lie-algebra will be a single extra technical complication in the proof of \cref{lemma: kill-B} below.
\end{rem}

\subsection{The image of the curvature operator}\label{subsec: curvature-image}
From \cref{prop: pointwise-TA} the O'Neill tensors $A, T$ satisfy a system of Clifford equations involving the image of $\mathcal R_N$, the curvature operator of $N$. Recall that $\mathcal R_N\geq0$ with a positive Ricci tensor. This section is a purely algebraic analysis of the consequences of these two positivity conditions, although many of the initial propositions do not use them to their full extent. The conclusions of this section will be applied to \cref{prop: pointwise-TA} in \cref{sec: curv-cliff}.

\begin{defi}
Let $p\in N$, we let $\overline{\mathfrak g}(p)\subseteq\mathfrak{so}(T_pN)$ denote the Lie algebra generated by the image of the curvature operator $\mathcal R_N$ under the identification $\Lambda^2T_pN\cong\mathfrak{so}(T_pN)$.
\end{defi}
First we decompose the action of $\overline{\mathfrak g}(p)$ on $T_pN$ into its irreducible summands. Since the Lie-algebra is generated by the image of a curvature operator, it will also decompose accordingly.
\begin{prop}\label{prop: curvature-decomp}
Let $p\in N$. Decomposing $T_pN = \bigoplus_i \overline{V}_i$ into irreducible summands of the $\overline{\mathfrak g}(p)$ action, one has
$$\overline{\mathfrak g}(p) =\bigoplus_i \overline{\mathfrak g}_i(p),$$
where $\overline{\mathfrak g}_i(p)$ is the Lie-algebra generated by $\mathcal R_N(\Lambda^2 \overline{V}_i)$. The Lie-algebra $\overline{\mathfrak{g}}_i(p)$ acts irreducibly on $\overline{V}_i$ and trivially on $\overline{V}_j$ for $j\neq i$.
\end{prop}
\begin{proof}
Note that $\Lambda^2 T_pN=\bigoplus_i \Lambda^2 \overline{V}_i \oplus \bigoplus_{i<j} \overline{V}_i\wedge \overline{V}_j$. We first verify that
\begin{equation}
\mathcal R_N(\Lambda^2T_pN)=\bigoplus_i\mathcal R_N(\Lambda^2 V_i),\label{eq: imR-decomp}
\end{equation}
i.e.\ $\mathcal R_N(X\wedge Y)=0$ for $X\in \overline{V}_i$, $Y\in \overline{V}_j$ with $i\neq j$. For any $\omega\in\Lambda^2T_pN$ one has
$$g_N( \mathcal R_N(X\wedge Y),\omega) =g_N(\mathcal R_N(\omega),X\wedge Y) =g_N( \mathcal R_N(\omega) Y,X)=0,$$
since $\overline V_i$ is invariant under the action of $\mathfrak g(p)$ and orthogonal to $\overline V_j$. The claim then follows.
\medbreak

Next we show that
\begin{equation}
\mathcal R_N(\Lambda^2 V_i)\subseteq \Lambda^2 V_i,\label{eq: imR-offdiag}
\end{equation}
i.e.\ $\mathcal R_N(X_1\wedge X_2)Y=0$ for $X_1,X_2\in V_i$, $Y\in \overline{V}_j$ with $i\neq j$. This is a consequence of the previous calculation and the Bianchi identity
$$\mathcal R_N(X_1\wedge X_2)Y=- \mathcal R_N(X_2\wedge Y)X_1-\mathcal R_N(Y\wedge X_1)X_2 =0.$$
The proposition then follows from equations (\ref{eq: imR-decomp}) and (\ref{eq: imR-offdiag}).
\end{proof}
Since $\overline{\mathfrak g}_i(p)$ acts irreducibly on $V_i$ and consists of anti-symmetric endomorphisms the (associative) algebra it generates in $\End(\overline{V}_i)$ is easy to understand. In principle four cases are possible by Schur's Lemma (cf.\ \cite{Rosenberg2016} for Schur's Lemma for real $*$-algebras). Non-degeneracy of the Ricci tensor at $p$ provides further restrictions:

\begin{prop}\label{prop: curvature-types}
Let $(\overline{\mathfrak{g}}_i(p), \overline{V}_i)$ be as in \cref{prop: curvature-decomp}. Then either:
\begin{enumerate}
\item \label{item1:prop-curvature-types}
The (associative) sub-algebra of $\End(\overline{V}_i)$ generated by $\overline{\mathfrak g}_i(p)$ is $\End(\overline{V}_i)$.
\item \label{item2:prop-curvature-types}
There is an $\R$-linear isometry $\iota\colon\overline{V}_i\to\C^k$ for some $k$. The (associative) algebra generated by $\overline{\mathfrak g}_i(p)$ is $\iota^*\End_\C(\C^k)$. Denoting with $I\colon\overline{V}_i\to \overline{V}_i$ the pullback of the complex unit of $\C^k$ one has $I\in\overline{\mathfrak g}_i(p)$.
\end{enumerate}
\end{prop}
\begin{rem}
\cref{prop: curvature-types} only requires non-degeneracy of the Ricci tensor of $\mathcal R_N$. Furthermore, if the image of the curvature operator is parallel then $\overline{\mathfrak g}(p)$ is the holonomy algebra of $N$ by a theorem of Ambrose and Singer \cite{ambrose-singer53}. \cref{prop: curvature-types} then reproduces well-known pointwise calculations, such as Calabi-Yau manifolds being Ricci-flat.
\end{rem}
\begin{proof}[Proof of \cref{prop: curvature-types}]
First note that $\overline{\mathfrak g}_i(p)=0$ immediately contradicts non-degeneracy of the Ricci tensor, since if $X\in \overline{V}_i$, and $e_j$ is some orthonormal basis of $T_pN$ then
$$\Ric(X,Y)=\sum_j g_N( \mathcal R_N(e_j\wedge Y)X,e_j) = 0$$
for all $Y\in T_pN$.

Let $\overline{\mathfrak g}_i(p)'$ denote the commutant of $\overline{\mathfrak g}_i(p)$ in $\End(\overline V_i)$, i.e.\ all endomorphisms commuting with $\overline{\mathfrak g}_i(p)$. Since $\overline{\mathfrak g}_i(p)\neq0$ consists of anti-symmetric endomorphisms and acts irreducibly on $\overline{V}_i$, we apply Schur's Lemma to find that $\overline{\mathfrak g}_i(p)'$ is isomorphic to one of the real division algebras $\R,\C,\numbers{H}$.

In the case $\R$ one recovers for the bi-commutant $\overline{\mathfrak g}_i(p)''=\End(\overline{V}_i)$. Now $\overline{\mathfrak g}_i(p)''$ is equal to the sub-algebra of $\End(\overline{V}_i)$ generated by $\overline{\mathfrak g}_i(p)$ by the bi-commutant theorem (which holds also for real $*$-algebras, see e.g.\ Theorem 4.3.8 of \cite{Li2003}). This case corresponds to point~(\ref{item1:prop-curvature-types}) in the proposition.

For the two remaining cases $\C,\numbers{H}$ one has $\overline{\mathfrak g}_i(p)''=\End_\C(\overline V_i)$ or $\overline{\mathfrak g}_i(p)''=\End_{\numbers H}(\overline V_i)$ with respect to some complex or quaternion structure on $\overline V_i$. In both cases $\C,\numbers{H}$ there is at least one complex unit $I\colon \overline V_i\to\overline V_i$ with $I\in\overline{\mathfrak g}_i(p)'$. The cases $\C,\numbers{H}$ can be distinguished by considering $\overline{\mathfrak g}_i(p)'\cap \overline{\mathfrak g}_i(p)''$, which is either $\C$ or $\R$ respectively.

We then proceed by assuming there is a complex unit $I\colon\overline V_i\to\overline V_i$ commuting with $\overline{\mathfrak g}_i(p)$. Using non-degeneracy of the Ricci tensor we will construct a non-zero anti-symmetric element of $\overline{\mathfrak g}_i(p)'\cap\overline{\mathfrak g}_i(p)''$, ruling out the case $\numbers{H}$. In fact our construction is such that $I\in\overline{\mathfrak g}_i(p)$. Carrying out these steps then completes the proof of the proposition.
\medbreak

So let $I\in\overline{\mathfrak g}_i(p)'$ be a complex unit, since the elements of $\overline{\mathfrak g}_i(p)$ commute with $I$ they are complex linear maps with respect to $I$. If $e_1,...,e_{2k}$ is a (real) orthonormal basis of $\overline{V}_i$, the \emph{complex trace} of an element $\omega\in\overline{\mathfrak g}_i(p)$ is given by:
$$2\Tr_\C(\omega)=\Tr_\R(\omega)-i\Tr_\R(I\omega)=\sum_j g_N( \omega e_j,e_j)+i \sum_jg_N(\omega e_j, Ie_j)\in\C.$$
Note that since $\omega$ is anti-symmetric the summand $\Tr_\R(\omega)$ will vanish. 
\medbreak

\noindent\textbf{Claim}: There is an element $\eta\in\overline{\mathfrak g}_i(p)$ for which $\Tr_\C(\eta)\neq0$.

To prove the claim let $X,Y\in\overline V_i(p)$ then
\begin{align}
2\Tr_\C(\mathcal R_N(X\wedge IY))&= i\sum_j g_N( \mathcal R_N(X\wedge IY)e_j,Ie_j)\\
& =-i\sum_j g_N(\mathcal R_N(X\wedge e_j)Ie_j,IY) - i\sum_j g_N(\mathcal R_N(X\wedge Ie_j)IY,e_j)
\end{align}
where we applied the Bianchi identity. Note that by assumption the linear isometry $I$ commutes with the image of $\mathcal R_N$, whence
$$-i\sum_j g_N (\mathcal R_N(X\wedge e_j)Ie_j,IY)=-i\Ric(X,Y).$$
Similarly using that $Ie_j$ is also an orthonormal basis of $\overline{V}_i$ one has
$$- i\sum_j g_N( \mathcal R(X\wedge Ie_j)IY,e_j)= i\sum_j g_N( \mathcal R_N(X\wedge I e_j)Y, Ie_j) = -i\Ric(X,Y),$$
and so $\Tr_\C(\mathcal R_N(X\wedge IY))=-i\Ric(X,Y)$, whence from non-degeneracy of $\Ric$ there is an $\eta\in\overline{\mathfrak g}_i(p)$ with $\Tr_\C(\eta)\neq0$. This proves the claim.
\medbreak

From this $\eta$ we now construct a non-vanishing anti-symmetric element of $\overline{\mathfrak g}_i(p)'\cap \overline{\mathfrak g}_i(p)''$. Let $G_i\subseteq U(\overline{V}_i)$ denote the Lie group generated by $\overline{\mathfrak g}_i(p)\subseteq\mathfrak u(\overline{V}_i)$. Then:
$$\Ad_{G_i}(\eta)\coloneq\frac1{\mathrm{Vol}(G_i)}\int_{G_i}\Ad_g(\eta)\intmathd \mu(g)$$
has the same complex trace as $\eta$ and is $\Ad$-invariant. It follows that it is non-vanishing and commutes with every element of $\overline{\mathfrak g}_i(p)$. This shows that it is an element of $\overline{\mathfrak{g}}_i(p)'$. Further it is a linear combination of elements $\Ad_g(\eta)\in\overline{\mathfrak g}_i(p)$ and hence is an element of $\overline{\mathfrak g}_i(p)\subseteq\overline{\mathfrak g}_i(p)''$. This completes the proof.
\end{proof}

We give the two cases from \cref{prop: curvature-types} names:
\begin{defi}\label{def: curvature-types}
Let $(\overline{\mathfrak{g}}_i(p), \overline{V}_i)$ be as in \cref{prop: curvature-decomp}. We say that
\begin{enumerate}
\item $(\overline{\mathfrak g}_i(p),\overline{V}_i)$ is of \emph{real type} if the (associative) algebra generated by $\overline{\mathfrak g}_i(p)$ is $\End(\overline{V}_i)$.
\item $(\overline{\mathfrak g}_i(p),\overline{V}_i)$ is of \emph{complex type} if there is an $\R$-linear isometry $\iota:\overline{V}_i\to\C^k$ for some $k$ and the (associative) algebra generated by $\overline{\mathfrak g}_i(p)$ is $\iota^*\End_\C(\C^k)$.
\end{enumerate}
\end{defi}

Beyond the information on the associative algebra generated by $\overline{\mathfrak g}_i(p)$ from \cref{prop: curvature-types} we will require one more fact on the representation theory of $\overline{\mathfrak g}_i(p)$. In contrast to the previous points, this fact will require both non-negativity of $\mathcal R_N$ as well as non-degeneracy of $\Ric_N$.

\begin{defi}
Let $(\overline{\mathfrak{g}}_i(p), \overline{V}_i)$ be as in \cref{prop: curvature-decomp}. We denote by $\overline{\mathfrak p}_i(p)$ the orthogonal complement of $\overline{\mathfrak g}_i(p)$ in $\Lambda^2\overline{V}_i$, which is closed under the adjoint action of $\overline{\mathfrak g}_i(p)$.
\end{defi}
\begin{prop}\label{prop: p-disjoint}
Let $(\overline{\mathfrak{g}}_i(p), \overline{V}_i)$ be as in \cref{prop: curvature-decomp}. Then $\overline{\mathfrak p}_i(p)$ does not contain any sub-representation of $\overline{\mathfrak g}_i(p)$ isomorphic to $\overline{V}_i$.
\end{prop}
\begin{rem}
\cref{prop: p-disjoint} is a combination of two results from the literature. In a paper introducing so called holonomy systems, Simons \cite{Simons1962} shows via pointwise computations that what we call $\overline{\mathfrak{g}}_i(p)$ is in fact the Lie algebra of an irreducible symmetric space. When dealing with a symmetric space, a computation of Laquer \cite{Laquer1975} then gives a Casimir element taking distinct values on $\overline{V}_i$ and $\overline{\mathfrak{p}}_i(p)$, showing they must be disjoint as representations.\end{rem}

For convenience of the reader we will provide a self-contained proof of \cref{prop: p-disjoint}. The proof will be carried out in several steps and is simpler than the just cited discussion. We first define an auxiliary curvature operator:

\begin{defi}
Let $G_i=\exp(\overline{\mathfrak{g}}_i(p))\subseteq O(\overline{V}_i)$ be the Lie group generated by $\overline{\mathfrak{g}}_i(p)$. We define
$$\widetilde R_i: \Lambda^2 \overline{V}_i\to\Lambda^2\overline{V}_i,\qquad \omega\mapsto\frac1{\mathrm{Vol}(G_i)}\int_{G_i}\Ad_g\brackets[\big]{\mathcal R_N(\Ad_{g^{-1}}\omega)}\intmathd \mu(g).$$
\end{defi}

\begin{lem}\label{lemma: curv-easy}
$\widetilde R_i$ is an algebraic curvature operator and
$$\image \widetilde R_i\subseteq \overline{\mathfrak{g}}_i(p)\subseteq \set[\big]{\omega\in\Lambda^2\overline V_i\mid \ad_\omega\circ\widetilde R_i = \widetilde R_i\circ\ad_\omega}.$$
\end{lem}
\begin{proof}
As a linear combination of curvature operators $\widetilde R_i$ is a curvature operator. Additionally $\overline{\mathfrak{g}}_i(p)$ is invariant under the action of $\Ad_g$ for all $g\in G_i$, whence $\Ad_g(\mathcal R_N(\Ad_{g^{-1}}\omega))\in\overline{\mathfrak{g}}_i(p)$ for all $\omega\in\Lambda^2\overline V_i$. This implies the first inclusion. The second inclusion follows from noting
$$\Ad_g \circ\widetilde R_i = \widetilde R_i \circ\Ad_g$$
for all $g\in G_i$ and taking the derivative by $g$.
\end{proof}
\begin{rem}
From the inclusion $\image \widetilde R_i\subseteq \{\omega\in\Lambda^2\overline V_i\mid \ad_\omega\circ\widetilde R_i = \widetilde R_i\circ\ad_\omega\}$ in \cref{lemma: curv-easy} one finds $\image \widetilde R_i$ is a Lie-algebra:
$$\sqbrackets[\big]{\widetilde R_i\omega,\widetilde R_i\eta}= \ad_{\widetilde R_i \omega}\brackets[\big]{\widetilde R_i\eta}=\widetilde R_i\brackets[\big]{\ad_{\widetilde R_i\omega}\eta}\in\image \widetilde R_i.$$
\end{rem}
\begin{lem}\label{lemma: Rtilde-irrep}
The Ricci tensor of $\widetilde R_i$ is non-degenerate and $\image \widetilde R_i$ acts irreducibly on $\overline{V}_i$.
\end{lem}
\begin{proof}
Since $\widetilde R_i$ is $G_i$-equivariant so is its Ricci tensor. But any quadratic form invariant under an irreducible orthogonal representation (such as $G_i$ on $\overline{V}_i$) must be proportional to the scalar product. We now verify this proportionality factor is non-zero:
Since the Ricci tensor of $\mathcal R_N$ is positive the partial trace of $\mathcal R_N$ over $\Lambda^2\overline{V}_i$ is strictly positive. By construction this partial trace is the trace of $\widetilde R_i$, whence $\widetilde R_i$ has non-vanishing (and hence non-degenerate) Ricci tensor.
\medbreak

We now verify that $\image \widetilde R_i$ acts irreducibly on $\overline{V}_i$. From \cref{prop: curvature-decomp} any two irreducible sub-representations of $\image \widetilde R_i$ are orthogonal to each other. If a sub-representation of $\image \widetilde R_i$ is not orthogonal to an irreducible summand it must then contain this summand.

Suppose $W\subset\overline{V}_i$ is a non-trivial irreducible sub-representation of $\image \widetilde R_i$, since $G_i$ acts irreducibly there is a $g\in G_i$ so that $gW\neq W$. By connectedness of $G_i$ we may assume $gW$ is not orthogonal to $W$. We will show $gW$ is a sub-representation of $\image \widetilde R_i$, whence it must contain $W$ by the previous paragraph. Since $W$ and $gW$ have the same dimension we conclude the contradiction $W=gW$. $\overline{V}_i$ then admits no non-trivial irreducible sub-representations and so must itself be irreducible.

We now show that $gW$ is a sub-representation of $\image \widetilde R_i$. Let $x\in W$, $\omega\in\Lambda^2\overline{V}_i$, unpacking the relation $\Ad_g( \widetilde R_i(\omega))=\widetilde R_i(\Ad_g\omega)$ gives:
$$\widetilde R_i(\omega)gx=g\widetilde R_i(\Ad_{g^{-1}}\omega)x\in gW$$
and so $gW$ is closed under the action of $\image \widetilde R_i$.
\end{proof}

\begin{lem}\label{lemma: Rtilde-invertible}
One has
$$\image \widetilde R_i= \overline{\mathfrak{g}}_i(p)= \set[\big]{\omega\in\Lambda^2\overline V_i\mid \ad_\omega\circ\widetilde R_i = \widetilde R_i\circ\ad_\omega}.$$
\end{lem}
\begin{proof}
By \cref{lemma: curv-easy} it is only necessary to show $\{\omega\in\Lambda^2\overline V_i\mid \ad_\omega\circ\widetilde R_i = \widetilde R_i\circ\ad_\omega\}\subseteq\image \widetilde R_i$. Let $\omega\in\image (\widetilde R_i)^\perp\cap\{\omega\in\Lambda^2\overline V_i\mid \ad_\omega\circ\widetilde R_i = \widetilde R_i\circ\ad_\omega\}$. Since $\widetilde R_i$ is self-adjoint we find $\widetilde R_i(\omega)=0$. Now for any $\eta\in\Lambda^2\overline V_i$:
$$\sqbrackets[\big]{\omega,\widetilde R_i\widetilde R_i\eta}=\ad_\omega\brackets[\big]{\widetilde R_i\widetilde R_i\eta}=\widetilde R_i\brackets[\big]{\ad_\omega(\widetilde R_i\eta)}=-\widetilde R_i\brackets[\big]{\ad_{\widetilde R_i\eta}\omega}=-\ad_{\widetilde R_i}\brackets[\big]{\widetilde R_i\omega}=0,$$
and so $\omega$ commutes with every element of $\image (\widetilde R_i^2)=\image \widetilde R$ and so lies in the commutant $\image (\widetilde R_i)'$. By \cref{lemma: Rtilde-irrep} we know that $\image \widetilde R_i$ acts irreducibly on $\overline V_i$ and that $\widetilde R_i$ has non-degenerate Ricci curvature. \cref{prop: curvature-types} then applies to show that all anti-symmetric elements of the commutant $\image (\widetilde R_i)'$ already lie in $\image \widetilde R_i$. We conclude $\omega=0$, which shows the lemma.
\end{proof}
\medbreak
From \cref{lemma: Rtilde-invertible} we see that restricting $\widetilde R_i$ to $\overline{\mathfrak{g}}_i(p)$ gives an invertible $\ad(\overline{\mathfrak{g}}_i(p))$-equivariant linear map $\overline{\mathfrak{g}}_i(p)\to\overline{\mathfrak{g}}_i(p)$. With this map we can now describe the Casimir element that distinguishes $\overline{V}_i$ from $\overline{\mathfrak{p}}_i(p)$.
\begin{lem}\label{lemma: casimir}
Let $\omega_1,...,\omega_k$ be an orthonormal basis of $\overline{\mathfrak{g}}_i(p)$. There is a scalar $\lambda_i<0$ so that:
\begin{enumerate}
\item \label{item1:lemma-casimir}
For any $v\in\overline V_i$ one has $\sum_\alpha \brackets[\big]{\widetilde R_i^{\sfrac{1}{2}}\omega_\alpha}^2v=\lambda_i\, v$.
\item \label{item2:lemma-casimir}
For any $\eta\in\overline{\mathfrak{p}}_i(p)$ one has $\sum_\alpha \brackets[\big]{\ad_{\widetilde R_i^{\sfrac{1}{2}}}\omega_\alpha}^2\eta=2\lambda_i\,\eta$.
\end{enumerate}
\end{lem}
\begin{proof}
Note that $\widetilde R^{\sfrac{1}{2}}_i\omega_1,...,\widetilde R^{\sfrac{1}{2}}_i\omega_k$ is an orthonormal basis for the (positive definite) $\ad(\overline{\mathfrak{g}}_i(p))$-invariant quadratic form
$$Q:\overline{\mathfrak{g}}_i(p)\times\overline{\mathfrak{g}}_i(p)\to\R, \qquad (\omega,\eta)\mapsto g_N\brackets[\big]{\omega, \widetilde R^{-1}_i\eta}.$$
It follows for any irreducible representation $\rho$ of $\overline{\mathfrak{g}}_i(p)$ that $\sum_\alpha \rho(\widetilde R_i^{\sfrac{1}{2}}\omega_\alpha)^2$ acts as a negative scalar, being the image of the Casimir element associated to $Q$. This shows point (\ref{item1:lemma-casimir}).
\medbreak

Suppose now that $\eta\in\overline{\mathfrak{p}}_i(p)$, recall by \cref{lemma: Rtilde-invertible} that this is equivalent to $\eta\in\Lambda^2\overline{V}_i$ and $\widetilde R_i(\eta)=0$. Choose an orthonormal basis $e_1,...,e_\ell$ of $\overline V_i$ so that $\eta=\sum_{ab}\eta_{ab} e_a\wedge e_b$. Then
\begin{align}
\sum_\alpha \brackets[\big]{\ad_{\widetilde R_i^{\sfrac{1}{2}}}\omega_\alpha}^2\eta&= \sum_{\alpha, ab}\eta_{ab} \!\left(\!\brackets[\big]{(\widetilde R_i^{\sfrac{1}{2}}\omega_\alpha)^2 e_a}\!\wedge\! e_b \!+\! e_a\!\wedge\! \brackets[\big]{(\widetilde R_i^{\sfrac{1}{2}}\omega_\alpha)^2e_b} \!-\! 2\brackets[\big]{\widetilde R_i^{\sfrac{1}{2}}\omega_\alpha e_a}\!\wedge\!\brackets[\big]{\widetilde R_i^{\sfrac{1}{2}}\omega_\alpha e_b}\!\right)\\
&=2\lambda_i\,\eta - 2\sum_{\alpha,ab}\eta_{ab}\,\brackets[\big]{\widetilde R_i^{\sfrac{1}{2}}\omega_\alpha e_a}\wedge\brackets[\big]{\widetilde R_i^{\sfrac{1}{2}}\omega_\alpha e_b}\\
&= 2\lambda_i\,\eta -2 \sum_{\alpha,abcd} \eta_{ab}\,g_N\brackets[\big]{\widetilde R_i^{\sfrac{1}{2}}\omega_\alpha e_a,e_c}\,g_N\brackets[\big]{\widetilde R_i^{\sfrac{1}{2}}\omega_\alpha e_b, e_d}\,e_c\wedge e_d\\
&=2\lambda_i\,\eta-\frac12\sum_{\alpha,abcd}\eta_{a,b}\, g_N\brackets[\big]{\widetilde R_i^{\sfrac{1}{2}}\omega_\alpha, e_a\wedge e_c}\,g_N\brackets[\big]{\widetilde R_i^{\sfrac{1}{2}}\omega_\alpha, e_b\wedge e_d}\,e_c\wedge e_d\\
&=2\lambda_i\,\eta-\frac12\sum_{abcd}\eta_{ab}\, g_N\brackets[\big]{\widetilde R_i(e_a\wedge e_c),e_b\wedge e_d}\,e_c\wedge e_d\label{eq: casimir}
\end{align}
where we used point (\ref{item1:lemma-casimir}) in the second line, as well as self-adjointness of $\widetilde R_i^{\sfrac{1}{2}}$ and the fact that $\omega_\alpha$ is an orthonormal basis of $\image \widetilde R_i$ in the third line. Since $\widetilde R_i$ satisfies the Bianchi identity we further recover
\begin{align}
g_N\brackets[\big]{\widetilde R_i(e_a\wedge e_c),e_b\wedge e_d} 
&=  g_N\brackets[\big]{\widetilde R_i(e_a\wedge e_b),e_c\wedge e_d}
- g_N\brackets[\big]{\widetilde R_i(e_a\wedge e_d),e_c\wedge e_b}.\label{eq: bianchi-Rtilde-eta}\end{align}
Recall that by definition $\widetilde R_i(\eta)=0$, which gives
$$\sum_{abcd}\eta_{ab}\,g_N\brackets[\big]{\widetilde R_i(e_a\wedge e_b),e_c\wedge e_d}\,e_c\wedge e_d = \widetilde R_i\brackets[\Big]{\sum_{ab}\eta_{ab}e_a\wedge e_b}=\widetilde R_i(\eta)=0$$
so that contracting \cref{eq: bianchi-Rtilde-eta} with $\sum_{abcd}\,\eta_{ab}\,e_c\wedge e_d$ will cancel the first term on the righthand side of \cref{eq: bianchi-Rtilde-eta}. Additionally using anti-symmetry $\eta_{ab}=-\eta_{ba}$ and self-adjointness of $\widetilde R_i$ one gets from \cref{eq: bianchi-Rtilde-eta}
\begin{align}
 \sum_{abcd} \eta_{ab}\,g_N\brackets[\big]{\widetilde R_i(e_a\wedge e_c),e_b\wedge e_d}\,e_c\wedge e_d
&=-\sum_{abcd}\eta_{ab}\,g_N\brackets[\big]{\widetilde R_i(e_a\wedge e_d),e_c\wedge e_b}\,e_c\wedge e_d\\
&=-\sum_{abcd} (-\eta_{ba})\,g_N\brackets[\big]{\widetilde R_i((-e_b)\wedge e_c),e_a\wedge e_d}\,e_c\wedge e_d\\
&=-\sum_{abcd} \eta_{ab}\,g_N\brackets[\big]{\widetilde R_i(e_a\wedge e_c),e_b\wedge e_d}\,e_c\wedge e_d.
\end{align}
The sum on the righthand side of \cref{eq: casimir} then vanishes. This completes the proof.
\end{proof}

We can now perform the proof of \cref{prop: p-disjoint}.

\begin{proof}[Proof of \cref{prop: p-disjoint}]
If $\overline{\mathfrak p}_i(p)$ would have an irreducible sub-representation isomorphic to $\overline{V}_i$ then there must be a non-zero $\eta\in\overline{\mathfrak p}_i(p)$ with $\sum_\alpha(\ad_{\widetilde R_i\omega_\alpha})^2\eta=\lambda_i\,\eta$ by point~(\ref{item1:lemma-casimir}) of \cref{lemma: casimir}. This contradicts point~(\ref{item2:lemma-casimir}) of the same lemma.
\end{proof}


\subsection{Interaction of the curvature operator and the Clifford algebra} \label{sec: curv-cliff}

Our goal is to use the information on $\overline{\mathfrak g}(f(p))$ derived in \cref{subsec: curvature-image} in order to conclude, using simple pointwise calculations, from \cref{eq: T-B-chain} that the shape operator $S_X$ is traceless for any $X\in\mathcal H_p$. This will be the last ingredient needed to show \cref{MainThm}, see \cref{sec:completion-proof}.

\begin{defi}\label{def: g-etc}
For $p\in M$ we let $\mathfrak g(p)$ denote the horizontal lift of $\overline{\mathfrak g}(f(p))$ to $p$. Similarly we let $V_i$, $\mathfrak g_i(p)$, $\mathfrak p_i(p)$ denote the horizontal lifts of $\overline V_i$, $\overline{\mathfrak g}_i(f(p))$, respectively $\overline{\mathfrak p}_i(f(p))$.
\end{defi}
When dealing with chains $\omega_1,...,\omega_\ell\in\mathfrak g(p)$ we can save a significant amount of bookkeeping by identifying them with elements of the tensor algebra generated by $\mathfrak g(p)$. To this end we make the following definition:

\begin{defi}Let $p\in M$.
\begin{enumerate}
\item We let $\mathfrak U(p)=\bigoplus_{k=1}^\infty \mathfrak g(p)^{\otimes k}$ denote the tensor algebra generated by $\mathfrak g(p)$. $\mathfrak U_i(p)\subseteq \mathfrak U(p)$ denotes the tensor algebra generated by $\mathfrak g_i(p)$.
\item We let $\ad:\mathfrak U(p)\to\End(\End(T_pM))$ denote the algebra-morphism defined by
$$\ad_{\omega_1\otimes...\otimes\omega_\ell}(\eta)=(\ad_{\omega_1}\circ\ldots\circ\ad_{\omega_\ell})\,(\eta)$$
for $\omega_1,...,\omega_\ell\in\mathfrak g(p)$.
\item We let $\rho:\mathfrak U(p)\to\End(T_pM)$ denote the morphism defined by
$$\rho(\omega_1\otimes...\otimes\omega_\ell)=\omega_1\cdots\omega_\ell$$
for $\omega_1,...,\omega_\ell\in\mathfrak g(p)$, where $\cdot$ is the product in $\End(T_pM)$.
\end{enumerate}
\end{defi}
\begin{rem}\label{rem: adjoint-rho}
Recall that $T_U^*$ and $A_X$ are maps $\mathcal H_p\to\mathcal V_p$ so that $\ad_\omega(T_U)=T_U\cdot\rho(\omega)^T$, $\ad_\omega(A_X)=A_X\cdot\rho(\omega)^T$ for $\omega\in\mathfrak U(p)$.
\end{rem}

We deal first with the real case.

\begin{prop}\label{prop: real-case}
For every $p\in M$ one has the following: Let $(\mathfrak g_i(p), V_i)$ be of real type and $h\in V_i$, then $A_Xh=0$, $T_Uh=0$ for all $X\in\mathcal H_p$ and $U\in\mathcal V_p$.
\end{prop}

We prove this in three steps. Recall \cref{eq: A-chain} from \cref{prop: pointwise-TA}, which gives for $\omega\in\mathfrak U(p)$ and $\eta$ a horizontal lift of an element of $\image\mathcal(R_N)_{f(p)}$ the estimate\footnote{Note that the constant $C$ depends on the length of the chain $\omega$, but this is inconsequential. We could e.g.\ absorb this dependence into the norm on $\mathfrak U(p)$ or note that we only need to work with chains up to a certain length, say $100\,\dim(M)^{100}$.}
$$\|c(\ad_\omega (2A_X\eta))\ue+(c-\overline c)(\ad_\omega B_X(\eta))\ue\|\leq C\|\omega\|\,\|\eta\|\,\|X\|\,\epsilon^r,$$
here $B_X$ was some element of $\End(\Lambda^2T_{f(p)}N)$. The goal of the first two steps is to show that the $B_X$ term can be dropped from this inequality. Here the main ingredient is \cref{prop: p-disjoint}. Letting $\pi_i:\mathcal H_p\to V_i$ denote the orthogonal projection to $V_i$ we first separate off all \say{off-diagonal} terms $\pi_\ell B_X(\eta)\pi_k$ for $\ell\neq k$:
\begin{lem}\label{lemma: kill-offdiag-B}
Let $p\in M$. For all $X\in\mathcal H_p$, $\omega\in\mathfrak U(p)$, $\eta$ a horizontal lift of an element of $\image\mathcal(R_N)_{f(p)}$, and any two distinct $(\mathfrak g_\ell(p),V_\ell)$, $(\mathfrak g_k(p), V_k)$ one has for all $\epsilon>0$
$$\|(c-\overline c)(\ad_\omega(\pi_\ell B_X(\eta)\pi_k))\ue\|\leq C\,\|\omega\|\,\|\eta\|\,\|X\|\,\epsilon^r.$$
\end{lem}
\begin{proof}
Let $\omega_\ell\in\mathfrak U_\ell(p)$, $\omega_k\in\mathfrak U_k(p)$ so that $\rho(\omega_\ell)=\pi_\ell$, $\rho(\omega_k)=\pi_\ell$. Then $\ad_{\omega_\ell\otimes\omega_k}(2A_X\eta)=2A_X\eta\pi_k\pi_\ell=0$ and  
$$ \ad_{\omega_\ell\otimes\eta_k}(B_X(\eta))=-\pi_kB_X(\eta)\pi_\ell-\pi_\ell B_X(\eta)\pi_k$$
since all mixed terms $\rho(\omega_{\ell,i})\rho(\omega_{k,j})$ vanish. Now
$$(c-\overline c)\,(\pi_k B_X(\eta)\pi_\ell)= (c-\overline c)\,(\pi_\ell B_X(\eta)\pi_k)$$
by anti-symmetry of $B_X(\eta)$, so one recovers for any $\omega\in\mathfrak U(p)$ that
\begin{align}
\|(c-\overline c)(\ad_\omega(\pi_\ell B_X(\eta)\pi_k))\ue\|&=\frac12 \|c(\ad_{\omega\otimes\omega_k\otimes\omega_\ell}(2A_X\eta))\ue+(c-\overline c)(\ad_{\omega\otimes\omega_k\otimes\omega_\ell}(B_X(\eta)))\ue\|\\
&\leq \frac12 C\,\|\omega\|\,\|\eta\|\,\|X\|\,\|\omega_\ell\|\,\|\omega_k\|\,\epsilon^r.
\end{align}
Absorbing the norms $\|\omega_\ell\|$, $\|\omega_k\|$ into the constant then proves the lemma.
\end{proof}
\begin{lem}\label{lemma: kill-B}
Let $p\in M$. For all $X\in\mathcal H_p$ and $\omega\in\mathfrak U(p)$ one has
\begin{equation}
\|c(A_X\cdot\rho(\omega)^T)\ue\|\leq C\,\|\omega\|\,\|X\|\,\epsilon^r.\label{eq: c(A_X)}
\end{equation}
\end{lem}
\begin{proof}
By \cref{lemma: kill-offdiag-B} one has
\begin{equation}
\|c(\ad_\omega(2A_X\eta))\ue+(c-\overline c)\,(\ad_\omega(\sum_j \pi_j B_X(\eta)\pi_j))\ue\|\leq C\,\|\omega\|\,\|\eta\|\,\|U\|\,\epsilon^r\label{eq: kill-B-1}
\end{equation}
for any $\omega\in\mathfrak U(p)$ and $\eta$ a horizontal lift of an element of $\image{(\mathcal R_N)_{f(p)}}$. Note first that since $A_X$ is a linear map $\mathcal H_p\to\mathcal V_p$, so
$$\ad_\omega(2A_X\eta)=2A_X\eta\cdot\rho(\omega)^T$$
by \cref{rem: adjoint-rho}. Next note that $\pi_jB_X(\eta)\pi_j$ is an element of $\Lambda^2 V_j=\mathfrak g_j(p)\oplus\mathfrak p_j(p)$, and recall that  $c-\overline c$ is $O(\epsilon^r)$ when applied to $\mathfrak g$ by \cref{eq: c-cbar-algebra} from \cref{prop: pointwise-TA}. Equation~(\ref{eq: kill-B-1}) then implies an equation of the form
$$\|c(2A_X\eta\cdot\rho(\omega)^T)\ue+(c-\overline c)\,(\ad_\omega\mathfrak p)\ue\|\leq C\,\|\omega\|\,\|\eta\|\,\|X\|\,\epsilon^r$$
for some element $\mathfrak p\in\bigoplus_j\mathfrak p_j(p)$. From \cref{prop: p-disjoint} we recall that the two representations of $\bigoplus_j\mathfrak g_j(p)$ appearing here, namely $\rho$ on $\mathcal H_p=\bigoplus_j V_j$ and $\ad$ on $\bigoplus_j\mathfrak p_j(p)$, are disjoint. By a general theorem of representation theory there is then an element $\omega'\in\mathfrak U(p)$ so that $\rho(\omega')=1$ and $\ad_{\omega'}(\mathfrak p)=0$. We remark that appealing to the general theorem is not really necessary, since one can also describe this element explicitly. If $\omega'(j)\coloneq\sum_\alpha (\widetilde R_j^{\sfrac{1}{2}}\omega_\alpha)^{\otimes 2}$ is the Casimir element of \cref{lemma: casimir}, then $\omega'=\sum_j \frac1{3\lambda_j^4}(4\lambda_j^2\omega'(j)^{\otimes2}-\omega'(j)^{\otimes ^4})$ does the job.
\medbreak

For any $\omega\in\mathfrak U(p)$ and $\eta$ any horizontal lift of an element of $\image(R_N)_{f(p)}$ we then have
\begin{align}
\|c(2A_X\cdot\rho(\omega\otimes\eta)^T)\ue\|&=\|c(2A_X\eta\cdot\rho(\omega\otimes \omega')^T)\ue+(c-\overline c)\,(\ad_{\omega\otimes\omega'}(\mathfrak p))\ue\|\\
&\leq C\,\|\omega\|\,\|\omega'\|\,\|\eta\|\,\|X\|\,\epsilon^r.
\end{align}
Absorbing $\|\omega'\|$ into the constant $C$ and $\eta$ into $\omega$ (recall that $\eta$ was arbitrary in the image of the curvature operator, which generates $\mathfrak g(p)$) then proves the lemma.
\end{proof}

With \cref{eq: c(A_X)} we can now proceed with the proof of \cref{prop: real-case}.

\begin{proof}[Proof of \cref{prop: real-case}]
Since $(\mathfrak g_i(p), V_i)$ is of real type $\rho(\mathfrak U_i(p))=\End(V_i)$ and hence there is an element $\omega_h\in\mathfrak U_i(p)$ so that $\rho(\omega_h)=hh^T$. Note that
$$ c(A_X\cdot hh^T) = c(A_Xh)\cdot c(h),$$
whence by \cref{eq: c(A_X)}
$$\|c(A_Xh)\|\,\|h\|=\| c(A_X\cdot hh^T)\ue\| \leq C \|\omega_h\|\,\|X\|\,\epsilon^r$$
for all $\epsilon>0$. In the limit $\epsilon\to0$ we recover $A_Xh=0$.
\medbreak

This then implies that $A_U^*h=0$ for all $U$, since $g(A_U^*h,X)=g(U,A_Xh)$, recall \cref{def: A-T-variants}. In particular $\ad_\omega (A_U^*)=0$ for all $\omega\in\mathfrak U_i(p)$. For $\omega\in\mathfrak U_i(p)$ \cref{eq: T-B-chain} from \cref{prop: pointwise-TA} then becomes
$$\|c(T_U^*\cdot\rho(\omega)^T)\ue\|=\|c(\ad_\omega(T_U^*))\ue\|\leq C\|\omega\|\,\|U\|\,\epsilon^r.$$
As before plugging in an $\omega_h\in\mathfrak U_i(p)$ for which $\rho(\omega_h)=hh^T$ and taking $\epsilon\to0$ shows $T_U^*h=0$.
\end{proof}

Now we turn to the complex case.

\begin{prop}\label{prop: complex-case}
Let $p\in M$ and $(\mathfrak g_i(p), V_i)$ be of complex type. Then $\Tr(S_h)=0$ for all $h\in V_i$ where $S_h$ is the shape operator of the fiber at $p$ of direction $h$.
\end{prop}

The initial preparation is similar to the real case, except that we can directly separate the equation for $T_U^*$ (that is \cref{eq: T-B-chain}) and do not need to consider the equation for $A_X$ (that is \cref{eq: A-chain}).

\begin{lem}
Let $p\in M$ and $(\mathfrak g_i(p),V_i)$ be of complex type. Then
\begin{equation}
\|c(\ad_\omega(T_U^*))\ue\|\leq C\,\|\omega\|\,\|U\|\,\epsilon^r\label{eq: c(T_U)-complex}
\end{equation}
for all $\omega\in\mathfrak U_i(p)$ and $\epsilon>0$.
\end{lem}
\begin{proof}
Recall \cref{eq: T-B-chain}:
$$\|c(\ad_\omega(2T_U^*+A_U))\ue\|\leq C\,\|\omega\|\,\|U\|\epsilon^r.$$
Identically to the proof of \cref{lemma: kill-offdiag-B} we can separate this to get the inequality
$$\|c(\ad_\omega(2T_U^*+\sum_j \pi_j A_U^*\pi_j))\ue\|\leq C\,\|\omega\|\,\|U\|\,\epsilon^r$$
for any $\omega\in\mathfrak U(p)$. Since $(\mathfrak g_i(p),V_i)$ is of complex type there is an $\omega_i\in\mathfrak U_i(p)$ with $\rho(\omega_i)=I$ a complex unit on $V_i$ (by definition, recall \cref{prop: curvature-types}). Note that $\ad_{\omega_i}(\pi_j A_U^*\pi_j)=0$ for $i\neq j$, and writing
$$L_1\coloneq \pi_i\frac{A_U^*+IA_U^*I}2\pi_i,\qquad L_2\coloneq \pi_i\frac{A_U^*-IA_U^*I}2\pi_i$$
one has $\pi_iA_U^*\pi_i=L_1+L_2$ as well as $(\ad_{\omega_i})^2(L_1)=-4L_1$, $(\ad_{\omega_i})^2(L_2)=0$. Then
\begin{align}
(\ad_{\omega_i})^2(2T_U^*+\sum_j \pi_jA_U^*\pi_j)=-2T_U^*\pi_i-4L_1,& &(\ad_{\omega_i})^4(2T_U^*+\sum_j \pi_j A_U^*\pi_j)=2T_U^*\pi_i+16 L_1.
\end{align}
It follows that for any $\omega\in\mathfrak U_i(p)$
$$\|c(\ad_\omega(6T_U^*))\ue\|=\Big\|c\Big(\ad_{\omega\otimes(4\omega_i^{\otimes2}+\omega_i^{\otimes4})}(2T_U^*+\sum_j\pi_j A_U^*\pi_j)\Big)\ue\Big\|\leq C\,\|\omega\|\,\|U\|\,\epsilon^r,$$
where we absorbed the norm of $4\omega_i^{\otimes2}+\omega_i^{\otimes4}$ into the constant. This proves the lemma.
\end{proof}

\begin{proof}[Proof of \cref{prop: complex-case}]
Recall that $\ad_\omega(T_U^*)=T_U^*\cdot \rho(\omega)^T$ by \cref{rem: adjoint-rho}. Since $(\mathfrak g_i(p), V_i)$ is of complex type one has by \cref{def: curvature-types} that the image of $\rho(\mathfrak U_i(p))$ is equal to $\End_\C(V_i)$, all linear maps commuting with a certain complex unit $I$ on $V_i$. In particular for $h\in V_i$ one has that
$$hh^T-Ihh^TI = hh^T +(Ih)(Ih)^T$$
lies in the image of $\rho$. One then finds
\begin{equation}
c(T_U^*h)c(h)\ue+c(T_U^*Ih)c(Ih)\ue=O(\epsilon^r),\label{eq: complex-case-1}
\end{equation}
where $\|\ue\|=1$ for all $\epsilon$. Our initial goal is to show that \cref{eq: complex-case-1} implies $\|T_U^*h\|=\|T_U^*Ih\|$ and $g(T_U^*h,T_U^*Ih)=0$. First we get from \cref{eq: complex-case-1}
$$\|T_U^*h\|\,\|h\|=\|c(T_U^*h)c(h)\ue\|=\|c(T_U^*Ih)c(Ih)\ue\|+O(\epsilon^r)=\|T_U^*Ih\|\,\|Ih\|+O(\epsilon^r).$$ 
Taking the limit $\epsilon\to0$ and using $\|h\|=\|Ih\|$ implies $\|T_U^*h\|=\|T_U^*Ih\|$. Knowing this we multiply \cref{eq: complex-case-1} on the left by $c(T_U^*h)c(Ih)+c(T_U^*Ih)c(h)$ and recover:
$$\left(\|T_U^*h\|^2 c(h)c(Ih)+\|Ih\|^2 c(T_U^*h)c(T_U^*Ih)+\|h\|^2c(T_U^*Ih)c(T_U^*h)+\|T_U^*Ih\|^2c(Ih)c(h)\right)\ue=O(\epsilon^r).$$
Since $\|T_U^*h\|=\|T_U^*Ih\|$, $\|h\|=\|Ih\|$, and $c(h)c(Ih)+c(Ih)c(h)=0$ we recover
$$-2\|h\|^2\scalarproduct{T_U^*h}{T_U^*Ih}\,\ue=O(\epsilon^r)$$
and so $T_Uh\perp T_UIh$. \medbreak
Recall that $S_hU=T_U^*h$ by definition and that $S_h$ is symmetric. So we have seen that for any $U\in\mathcal V_p$ and $h\in V_i$ one has
\begin{gather}
\|S_h U\|^2 = \|T_U^*h\|^2 = \|T_U^*Ih\|^2 = \|S_{Ih}U\|^2,\\
g(U,(S_hS_{Ih}+S_{Ih}S_h)U)=2g(S_hU, S_{Ih}U)=g(T_U^*h,T_{U}^*Ih)=0.
\end{gather}
These relations imply $\Tr(S_h)=0$, since e.g.\ upon complexifying $TM\otimes\C$ with a complex unit $i$ one computes
$$(S_h+iS_{Ih})^2=S_h^2-S_{Ih}^2+i\,(S_hS_{Ih}+S_{Ih}S_{h})=0.$$
Then $S_h+iS_{Ih}$ is nilpotent and so traceless, but its trace is just $\Tr(S_h)+i\Tr(S_{Ih})$.
\end{proof}

\section{Completing the proof of \cref{MainThm} and its applications} \label{sec:completion-proof}

In this section we first describe how the results of the previous sections combine to prove the main result of this paper, \cref{MainThm}. After this we prove how \cref{MainThm} implies \cref{cor:Llarull-type-general}.

\begin{proof}[Proof of \cref{MainThm}]
Recall that $f\colon(M,g_M)\to(N,g_N)$ is a Riemannian submersion by \cref{ThmRiemSubmersion}. By \cref{PropRicUp=RicDown} we additionally have that $\Ric_M=f^*\Ric_N$. From \cref{cor: ricci-minimal-fibers} we see that $f$ is locally a Riemannian product, i.e.\ the O'Neill tensors $A$ and $T$ vanish identically, if the mean curvature of the fibers vanish. We now show how the vanishing of the mean curvature follows from the results of the previous sections.\medbreak

Let $p\in M$ and $X\in T_pM$ be a horizontal vector. Decompose $X=\sum_i h_i$ for $h_i\in V_i$, where $(\mathfrak g_i(p),V_i)$ is the decomposition of $\mathcal H_p$ of \cref{subsec: curvature-image,sec: curv-cliff}, recall \cref{def: g-etc} and \cref{prop: curvature-decomp}. Then $S_{X}=\sum_i S_{h_i}$, where $S$ is the shape operator of the fibers, recall \cref{def: A-T-variants}. If $(\mathfrak g_i(p), V_i)$ is of real type one has $S_{h_i}=0$ by \cref{prop: real-case}, and if $(\mathfrak g_i(p), V_i)$ is of complex type one has $\Tr(S_{h_i})=0$ by \cref{prop: complex-case}. Combining both cases gives $\Tr(S_X)=\sum_i \Tr(S_{h_i})=0$. Since $X$ was arbitrary the mean curvature of the fibers $H=\Tr(S_{(\cdot)})^\sharp$ vanishes.
\end{proof}\medbreak

We now turn to the proof of \cref{cor:Llarull-type-general}. \medbreak

\begin{proof}[Proof of \cref{cor:Llarull-type-general}]
  The proof strategy is to apply \cref{MainThm} to the spin map $\pr_1\circ f\colon M\to N$, hence we have to verify that $\chi\brackets{N} \cdot \hideg\brackets{\pr_1\circ f}\neq 0$. Let~$p\in N$ be a regular value of $\pr_1\circ f$ such that there exists a point $q\in F$ for which $\brackets{p,q}$ is a regular value of~$f$.  
  We define
  \begin{equation}
    M_p\coloneqq f^{-1}\brackets{\set{p}\times F}=\brackets{\pr_1\circ f}^{-1}\brackets{p} \quad \text{and} \quad f_p\coloneqq \restrict{\brackets{\pr_2\circ f}}{M_p}\colon M_p\to F.
  \end{equation}
  Note that $\deg\brackets{f_p}=\deg\brackets{f}\neq 0$.
  %
  %
  Let $\mathcal{L}\brackets[\normalsize]{F}$ be the Mishchenko--Fomenko bundle of the manifold~$F$ for the reduced group $\Cstar$-algebra of~$\pi_1\brackets{F}$. We define $E\coloneqq\brackets{\pr_2\circ f}^*\mathcal{L}\brackets[\normalsize]{F}$, which is a flat bundle of finitely generated projective Hilbert $\Cstar \pi_1\brackets{F}$-modules over the manifold~$M$, and claim that
  \begin{equation} \label{eq:proof-cor-Llarull-type-2}
    \ind\brackets[\big]{\Dirac_{\SpinBdl M_p\otimes \restrict{E}{M_p}}} \otimes_{\Z} \Q\neq 0 \in \KO_{\dim\brackets{F}}\brackets{\Cstar \pi_1\brackets{F}} \otimes_{\Z} \Q.
  \end{equation} \medbreak
  We argue as in the proof of \cite[Proposition 5.2]{Zeidler2020}. For some discrete group~$\Gamma$, its classifying space~$B\Gamma$ and some closed spin manifold~$X$, the \textit{Novikov assembly map} $\nu\colon \KO_*\brackets{B\Gamma} \to \KO_*\brackets{\Cstar \Gamma}$ satisfies $\nu\brackets{\phi_*\sqbrackets{X}_{\KO}}=\ind\brackets[\big]{\Dirac_{\SpinBdl X\otimes \phi^*\mathcal{L}\brackets{B\Gamma}}}$ for every map $\phi\colon X\to B\Gamma$ \cite[cf.][Section~2]{Zeidler2020}. Let~$\Gamma=\pi_1\brackets{F}$, ~$X=M_p$, and $\phi=\mu\circ f_p$. Here $\mu\colon F\to B\pi_1\brackets{F}$ denotes the classifying map of the universal cover of~$F$. As in \cref{rem:non-trivial-mapping-fundamental-class} we denote by~$\nu$ the rational Novikov assembly map and by~$\ph$ the Pontryagin character, and obtain
  \begin{equation} \label{eq:proof-cor-Llarull-type-3}
    \begin{split}
      \brackets{\nu\circ \mu_*\circ \ph^{-1}\circ \brackets{f_p}_*}\brackets[\big]{\sqbrackets{M_p}\cap \hat{A}\brackets{M_p}}
      &=\brackets{\nu\circ \brackets{\mu\circ f_p}_*\circ \ph^{-1}}\brackets[\big]{\sqbrackets{M_p}\cap \hat{A}\brackets{M_p}} \\
      &=\ind\brackets[\big]{\Dirac_{\SpinBdl M_p\otimes \restrict{E}{M_p}}}\otimes_\Z \Q.
    \end{split}
  \end{equation}
  Here we used $\ph\brackets{\sqbrackets{M_p}_{\KO}}=\sqbrackets{M_p}\cap \hat{A}\brackets{M_p}$. Since $\deg\brackets{f_p}\neq 0$, the top degree part of $\brackets{f_p}_*\brackets[\big]{\sqbrackets{M_p}\cap \hat{A}\brackets{M_p}}$ is non-trivial in $H_*\brackets{F;\Q}$. By the assumption on the manifold~$F$ and \cref{rem:non-trivial-mapping-fundamental-class}, it follows that the image of $\brackets{f_p}_*\brackets[\big]{\sqbrackets{M_p}\cap \hat{A}\brackets{M_p}}$ under the map  $\nu\circ \mu_*\circ \ph^{-1}$ is non-trivial, hence \cref{eq:proof-cor-Llarull-type-2} holds by \cref{eq:proof-cor-Llarull-type-3}. \medbreak
  By \cite[Remark~5.7]{Tony2025a}, the condition on the higher mapping degree in \cref{EqIndexTheoreticConditionThmA} is fulfilled, and \cref{MainThm} is applicable to the map $\pr_1\circ f\colon M\to N$. We obtain that there exists a closed connected Ricci-flat spin manifold~$\overline{F}$ such that~$M$ is locally isometric to the Riemannian product $N\times \overline{F}$ and $\pr_1\circ f$ is given by the projection onto the first factor. By the structure theorem for Ricci-flat manifolds \cite[cf.][Theorem~4.5]{Fischer1975}, the Ricci-flat manifold~$\overline{F}$ is on a finite Riemannian covering isometric to the Riemannian product of some simply-connected Ricci-flat manifold~$X$ and some flat torus. Since~$F$ is rationally essential and the map $f_p\colon \overline{F}\cong M_p\to F$ has non-zero degree, we obtain that~$X$ must be a point and the corollary is proved.
\end{proof}
%
%
\printbibliography
\end{document}